\documentclass{article}

\usepackage{lineno} \modulolinenumbers[5]
\usepackage[a4paper]{geometry}
\usepackage{amsfonts}
\usepackage{parskip}
\usepackage{xtab}
\usepackage{setspace}
\usepackage{fancyhdr}
\usepackage[dvips]{graphicx}
\usepackage[T1]{fontenc}
\usepackage[latin1]{inputenc}
\usepackage[main=english,italian]{babel}
\usepackage{type1ec}
\usepackage[final]{microtype}
\usepackage{amsmath,amssymb,amsthm,eucal,dsfont,bm,mathrsfs,stmaryrd}
\usepackage{lmodern}
\usepackage{textcomp}
\usepackage{pict2e,enumitem}
\usepackage{hyperref}
\usepackage[nodayofweek]{datetime} % change the format of printed dates (no american style!)
\usepackage{comment}

\hypersetup{
    pdfpagemode={UseOutlines},
    bookmarksopen,
    pdfstartview={FitH},
    colorlinks,
    linkcolor={black},
    citecolor={black},
    urlcolor={black}
            }

\usepackage{geometry}

\newdateformat{monthyear}{\monthname[\THEMONTH], \THEYEAR}% new date format

\DeclareMathOperator{\Id}{Id}

\DeclareMathOperator{\rank}{rank}

\newcommand{\N}{\ensuremath{\mathbb{N}}}% natural numbers
\newcommand{\R}{\ensuremath{\mathbb{R}}}% real numbers
\theoremstyle{definition}
        
        \newtheorem*{remark}{Remark}
        \newtheorem*{example}{Example}
        
\theoremstyle{plain}
        \newtheorem{theorem}{Theorem}
        \newtheorem{lemma}{Lemma}
        \newtheorem{corollary}[theorem]{Corollary}
        \newtheorem{prop}{Proposition}

\title{\textbf{Schauder estimates for degenerate L\'evy \\
Ornstein-Ulhenbeck operators}}
\author{\textbf{Lorenzo Marino}\footnote{Laboratoire de Mod\'elisation Math\'ematique d'Evry (LaMME), Universit\'e d'Evry Val d'Essonne, $23$
Boulevard de France $91037$ Evry, France and Dipartimento di Matematica, Universit\`a di Pavia, Via Adolfo Ferrata $5$, $27100$ Pavia,
Italy.\newline
\emph{E-mail Adress:} lorenzo.marino@univ-evry.fr}\phantom{,} \!\footnote{This work was supported by a public grant ($2018-0024H$)
as part of the FMJH project.}}

\begin{document}
\maketitle
\begin{abstract}
We establish global Schauder estimates for integro-partial differential equations (IPDE) driven by a possibly degenerate
L\'evy Ornstein-Uhlenbeck operator, both in the elliptic and parabolic setting, using some suitable anisotropic H\"older spaces. The class
of operators we consider is composed by a linear drift plus a L\'evy operator that is comparable, in a suitable sense, with a
possibly truncated stable operator. It includes for example, the relativistic, the tempered, the layered or the Lamperti
stable operators. Our method does not assume neither the symmetry of the L\'evy operator nor the invariance for dilations
of the linear part of the operator. Thanks to our estimates, we prove in addition the well-posedness of the considered IPDE in  suitable
functional spaces.
In the final section, we extend some of these results to more general operators involving non-linear, space-time dependent
drifts.
\end{abstract}


{\small{\textbf{MSC:} Primary: $35$J$70$, $35$K$65$, $35$R$09$, $35$B$45$; Secondary: $60$H$30$.}}

\section{Introduction}
Fixed an integer $N$ in $\N$, we consider the following integro-partial differential operator of Ornstein-Uhlenbeck type:
\begin{equation}
\label{Degenerate_Stable_PDE}
  \mathcal{L}^{\text{ou}} \, := \, \mathcal{L}+ \langle A x , D_x\rangle \quad \text{ on } \R^N,
\end{equation}
where  $\langle \cdot,\cdot \rangle$ denotes the Euclidean inner product on $\R^N$, $A$ is a matrix in $\R^N\otimes \R^N$ and $\mathcal{L}$ is a
possibly degenerate, L\'evy operator acting non-degenerately only on a subspace of $\R^N$. We are interested in showing the
\emph{well-posedness} and the associated \emph{Schauder estimates} for elliptic and parabolic equations involving the operator
$\mathcal{L}^{\text{ou}}$ and with coefficients in a generalized family of H\"older spaces. \newline
We only assume that $A$ satisfies a natural controllability assumption, the so-called Kalman rank condition (condition
[\textbf{K}] below), and that the operator $\mathcal{L}$ is comparable, in a suitable sense, to a non-degenerate,
truncated $\alpha$-stable operator on the same subspace of $\R^N$, for some $\alpha<2$ (condition [\textbf{ND}]
below).\newline
The topic of Schauder estimates for Ornstein-Uhlenbeck operators has been widely studied in the last decades, especially in the diffusive,
local setting, i.e.\ when $\mathcal{L}=\frac{1}{2}\text{Tr}\bigl(Q D^2_x\bigr)$ for some suitable matrix $Q$, and it is now
quite well-understood. See e.g.\ \cite{Daprato:Lunardi95}, \cite{Lunardi97}, \cite{book:Gilbarg:Trudinger01}, \cite{DiFrancesco:Polidoro06},
\cite{Krylov:Priola10}, \cite{Chaudru:Honore:Menozzi18_Sharp}. \newline
On the other hand, a literature on the topic for the pure jump, non-local framework has been developed only in the recent years
(\cite{Bass09}, \cite{Dong:Kim13}, \cite{Bae:Kassmann15}, \cite{Ros-Oton:Serra16}), \cite{Fernanadez:Ros-Oton17},
\cite{Imbert:Jin:Shvydkoy18}, \cite{Chaudru:Menozzi:Priola19}, \cite{Kuhn19},  but mainly
in the non-degenerate, $\alpha$-stable setting, i.e.\ when $\mathcal{L}=\Delta^{\alpha/2}_x$ is the fractional Laplacian on $\R^N$ or
similar. Up to the best of our knowledge, the only two articles dealing with the degenerate, non-local framework (if
$\mathcal{L}=\Delta^{\alpha/2}_x$ acts non-degenerately only on a sub-space of $\R^N$) are
\cite{Hao:Peng:Zhang19}, that takes into account the kinetics dynamics ($N=2d$), and \cite{Marino20}, for the general chain. In
order to use \cite{Hao:Peng:Zhang19} or \cite{Marino20} for our operator \eqref{Degenerate_Stable_PDE}, we would need to impose the
additional strong assumption of invariance for dilations of the matrix $A$.\newline
Even if the Ornstein-Uhlenbeck operator is usually seen as a "toy model" for more general operators with space-time
dependent, non-linear coefficients, we highlight that they appear naturally in various scientific contexts: for example in physics, for the
analysis of anomalous diffusions phenomena or for Hamiltonian models in a turbulent regime (see e.g.
\cite{Baeumer:Benson:Meerschaert01}, \cite{Cushman:Park:Kleinfelter:Moroni05} and the references therein) or in mathematical finance and
econometrics (see e.g.\ \cite{Brockwell01}, \cite{Barndorff-Nielsen:Shephard01}). \newline
The interest in Schauder estimates involving this type of operator also follows from the natural application which consists in establishing
the well-posedness of stochastic differential equations (SDE) driven by L\'evy processes and the associated stochastic control theory. See
e.g.\ \cite{Fleming:Mitter82}, \cite{Chaudru:Menozzi17_Weak}, \cite{Hao:Wu:Zhang19}.\newline
Under our assumptions, we have been able to consider more general L\'evy operators not usually included in the literature, such as the
relativistic stable process, the layered stable process or the Lamperti one (see Paragraph "Main Operators Considered" below for details).
Moreover, we do not require the operator $\mathcal{L}$ to be symmetric.\newline
Here, we only mention one important example that satisfies our hypothesis, the Ornstein-Uhlenbeck operator on $\R^2$ driven by the
relativistic fractional Laplacian $\Delta^{\alpha/2}_{\text{rel}}$ and acting only on the first component:
\begin{equation}\label{eq:example_new}
x_1(D_{x_1}\phi(x)+D_{x_2}\phi(x)) + \text{p.v.}\int_{\R}\bigl[\phi(\begin{pmatrix} x_1+z \\ x_2 \end{pmatrix})-\phi(\begin{pmatrix} x_1
\\ x_2 \end{pmatrix})\bigr]\frac{1+\vert z \vert^{\frac{d+\alpha -1}{2}}}{\vert z \vert^{d+\alpha}}e^{-\vert z \vert}\, dz \, = \, \langle A
x , D_x\phi(x) \rangle + \mathcal{L}\phi(x)
\end{equation}
where $x=(x_1,x_2)$ in $\R^2$. Such an example is included in the framework of Equation \eqref{Degenerate_Stable_PDE}
considering $A=\begin{pmatrix}1 & 0 \\ 1 & 0 \end{pmatrix}$. This operator appears naturally as a fractional
generalization of the relativistic Schr\"odinger operator (See \cite{Ryznar02} for more details).\newline
We remark that example \eqref{eq:example_new} cannot be considered in \cite{Hao:Peng:Zhang19} or in our previous work
\cite{Marino20}. Indeed, the matrix
$A_0$ is not "dilation-invariant" (see example \ref{example_basic} below) and thus, it cannot be rewritten in the form used in
\cite{Marino20} (see also \cite{Lanconelli:Polidoro94} Proposition $2.2$ for a more thorough explanation). Furthermore, operators like the
relativistic fractional Laplacian cannot be treated  in \cite{Hao:Peng:Zhang19} or \cite{Marino20} that indeed have taken into account only
stable-like operators on $\R^N$. Another useful advantage of our technique is that we do not need anymore the symmetry of
the L\'evy measure $\nu$ which was, again, a key assumption in \cite{Marino20}.
%%%%%%%%%%%%%%%%%%%%%%%%%%%%%%%%%%%%%%%%%%%%%%%%%%%%%%%%%%%%%%%%%%%%%%%%%%%%%%%%%%%%%%%%%%%%%%%%%%%%%%%%%%%%%%%%%%%%%%%%%%%%%%%%%%%%%%%

More in details, given an integer $d\le N$ and a matrix $B$ in $\R^N\otimes \R^d$ such that $\text{rank}(B)=d$, we consider a
family of  operators $\mathcal{L}$ that can be represented for any sufficiently regular function $\phi\colon \R^N\to \R$ as
\begin{equation}\label{eq:def_operator_L}
\mathcal{L}\phi(x)\, := \, \frac{1}{2}\text{Tr}\bigl(BQB^\ast D^2_x\phi(x)\bigr) +\langle Bb,D_x \phi(x)\rangle +
\int_{\R^d_0}\bigl[\phi(x+Bz)-\phi(x)-\langle D_x\phi(x), Bz\rangle\mathds{1}_{B(0,1)}(z) \bigr] \,\nu(dz),
\end{equation}
where $b$ is a vector in $\R^d$, $Q$ is a symmetric, non-negative definite matrix in $\R^d\otimes \R^d$ and $\nu$ is a L\'evy measure on
$\R^d_0:=\R^d\smallsetminus \{0\}$, i.e.\ a $\sigma$-finite measure on $\mathcal{B}(\R^d_0)$, the Borel $\sigma$-algebra on $\R^d_0$, such
that $\int(1\wedge \vert z\vert^2) \, \nu(dz)$ is finite. We then suppose $\nu$ to satisfy the following \emph{non-degeneracy condition}:
\begin{description}
  \item[{[ND]}] there exists $r_0>0$, $\alpha$ in $(0,2)$ and a finite, non-degenerate measure $\mu$ on the unit sphere $\mathbb{S}^{d-1}$
      such that
      \[\nu(C) \, \ge \, \int_{0}^{r_0}\int_{\mathbb{S}^{d-1}} \mathds{1}_{C}(r\theta)\, \mu(d\theta)\frac{dr}{r^{1+\alpha}}, \quad
      C\in \mathcal{B}(\R^d_0).\]
\end{description}
We recall that a measure $\mu$ on $\R^d$ is non-degenerate if there exists a constant $\eta \ge 1$ such that
\begin{equation}\label{eq:non_deg_measure}
\eta^{-1}\vert p \vert^\alpha \, \le \, \int_{\mathbb{S}^{d-1}}\vert p\cdot s \vert^\alpha \, \mu(ds) \, \le\,\eta
\vert p \vert^\alpha, \quad p \in \R^d,
\end{equation}
where $"\cdot"$ stands for the inner product on the smaller space $\R^d$. Since any $\alpha$-stable L\'evy measure $\nu_\alpha$  can be
decomposed into a spherical part $\mu$ on
$\mathbb{S}^{d-1}$ and a radial part $r^{-(1+\alpha)}dr$ (see e.g.\ Theorem $14.3$ in \cite{book:Sato99}), assumption
[\textbf{ND}] roughly states that the L\'evy measure of the integro-differential part of $\mathcal{L}$ is bounded from below by the L\'evy
measure of a possibly truncated, $\alpha$-stable operator on $\R^d$. \newline
It is assumed moreover that the matrixes $A$, $B$ satisfy the following \emph{Kalman condition}:
\begin{description}
  \item[{[K]}] It holds that $N \, = \, \text{rank}\bigl[B,AB,\dots,A^{N-1}B\bigr]$,
\end{description}
where $\bigr[B,AB,\dots,A^{N-1}B\bigr]$ is the matrix in $\R^N\otimes\R^{dN}$ whose columns are $B,AB,\dots,A^{N-1}B$. \newline
Such an assumption is equivalent, in the linear framework, to the  H\"ormander condition (see \cite{Hormander67}) on the
commutators, ensuring the hypoellipticity of the operator $\partial_t-\mathcal{L}^{\text{ou}}$. Moreover, condition [\textbf{K}] is
well-known in control theory (see e.g.\ \cite{book:Zabczyk95}, \cite{Priola:Zabczyk09}).

\paragraph{Mathematical Outline.} In the present paper, we aim at establishing global Schauder estimates for equations involving the operator
$\mathcal{L}^{\text{ou}}$ on $\R^N$, both in the elliptic and parabolic settings. Namely, we consider for a fixed $\lambda>0$ the following
elliptic equation:
\begin{equation}\label{eq:Elliptic_IPDE}
\lambda u(x) - \mathcal{L}^{\text{ou}}u(x) \, = \, g(x), \quad x \in \R^N,
\end{equation}
and, for a fixed time horizon $T>0$, the following parabolic Cauchy problem:
\begin{equation}\label{eq:Parabolic_IPDE}
\begin{cases}
  \partial_tu(t,x) \, = \, \mathcal{L}^{\text{ou}}u(t,x)+f(t,x), \quad (t,x) \in (0,T)\times \R^N; \\
  u(0,x) \, = \, u_0(x), \quad x \in \R^N,
\end{cases}
\end{equation}
where $f,g,u_0$ are given functions. Since our aim is to show optimal regularity results in H\"older spaces, we will assume for the elliptic
case (Equation \eqref{eq:Elliptic_IPDE}) that the source $g$ belongs to a suitable \emph{anisotropic}
H\"older space $C^\beta_{b,d}(\R^N)$ for some $\beta$ in $(0,1)$, where the H\"older exponent depends on the "direction" considered. The space
$C^\beta_{b,d}(\R^N)$ can be understood as composed by the bounded functions on $\R^N$
that are H\"older continuous with respect to a distance $\rho$ somehow induced by the operator $\mathcal{L}^{\text{ou}}$. We refer to
Section $2$ for a detailed exposition of such an argument but we
highlight already that the above mentioned distance $d$ can be seen as a generalization of the classical parabolic distance, adapted to our
degenerate, non-local framework. It is precisely assumption [\textbf{K}], or equivalently the hypoellipticity of
$\partial_t+\mathcal{L}^{ou}$, that ensures the existence of such a distance $d$ and gives it its anisotropic nature. Roughly speaking, it
allows the smoothing effect of the L\'evy operator $\mathcal{L}$ acting non-degenerately only on some components, say $B\R^N$, to spread in
the whole space $\R^N$, even if with lower regularizing properties. \newline
Concerning the parabolic problem \eqref{eq:Parabolic_IPDE}, we assume similarly that $u_0$ is in $C^{\alpha+\beta}_{b,d}(\R^N)$ and that
$f(t,\cdot)$ is in $C^{\beta}_{b,d}(\R^N)$, uniformly in $t\in(0,T)$. The typical estimates we want to prove can be stated in the parabolic
setting in the following way: there exists a constant $C$, depending only on the parameters of the model, such that any
distributional solution $u$ of the Cauchy problem \eqref{eq:Parabolic_IPDE} satisfies
\begin{equation}
\label{eq:SchauderEstimatesintro} \tag{$\mathscr{S}$}
\Vert u \Vert_{L^\infty(C^{\alpha+\beta}_{b,d})} \, \le \, C\bigl[\Vert u_0 \Vert_{C^{\alpha+\beta}_{b,d}}+\Vert f
\Vert_{L^\infty(C^{\beta}_{b,d})} \bigr].
\end{equation}
As a by-product of the Schauder Estimates $(\mathscr{S})$, we will obtain the well-posedness of the
Cauchy problem \eqref{eq:Parabolic_IPDE} in the space $L^\infty\bigl(0,T;C^{\alpha+\beta}_{b,d}(\R^N)\bigr)$ , once the existence of a
solution is established. The additional regularity for the solution $u$ with respect to the source $f$ reflects the appearance of a
smoothing effect associated with $\mathcal{L}^{\text{ou}}$ of order $\alpha$, as it is expected by condition [\textbf{ND}]. It can be seen
as a generalization of the "standard" parabolic bootstrap to our degenerate, non-local setting. We highlight that the parabolic
bootstrap in ($\mathscr{S}$) is precisely derived from the non-degenerate stable-like part in $\mathcal{L}$ (lowest regularizing effect in
the operator).  \newline
To show our result, we will follow the semi-group approach as firstly introduced in \cite{Daprato:Lunardi95}, which became
afterwards a very robust tool to study Schauder estimates in a wide variety of frameworks (\cite{Lunardi97}, \cite{Lorenzi05}, \cite{Saintier07}, \cite{Priola09}, \cite{Priola12}, \cite{Dong:Kim13}, \cite{Kim:Kim15}, \cite{Chaudru:Honore:Menozzi18_Sharp}, \cite{Kuhn19}). The main idea is to consider the Markov transition semigroup $P_t$ associated with
$\mathcal{L}^{ou}$ and then, in the elliptic case, to use the Laplace transform
formula in order to represent the unique distributional solution $u$ of Equation \eqref{eq:Elliptic_IPDE} as:
\[u(x) \, = \, \int_{0}^{\infty}e^{-\lambda t}\bigl[P_tg\bigr](x) \, dt \, =: \, \int_{0}^{\infty}e^{-\lambda t}P_tg(x) \, dt, \quad x \in
\R^N.\]
In the parabolic setting, we exploit instead the variation of constants (or Duhamel) formula in order to show a similar representation for
the weak solution of the Cauchy problem \eqref{eq:Parabolic_IPDE}:
\[u(t,x) \, = \, P_tu_0(x) +\int_{0}^{t}\bigl[P_{t-s}f(s,\cdot)\bigr](x)\, ds \, =: \, P_tu_0(x) +\int_{0}^{t}P_{t-s}f(s,x)\, ds, \quad t
\in [0,T], \, x \in \R^N.\]
In order to prove global regularity estimates for the solutions, the crucial point is to understand the
action of the operator $P_t$ on the anisotropic H\"older spaces. In particular, we will show in Corollary
\ref{corollary:continuity_between_holder} the continuity of $P_t$
as an operator from $C^\beta_{b,d}(\R^N)$ to $C^\gamma_{b,d}(\R^N)$ for $\beta<\gamma$ and, more precisely, that it holds:
\begin{equation}\label{eq:Intro:Gradient_Estimates}
\Vert P_t \phi\Vert_{C^\gamma_{b,d}} \, \le \,C\Vert \phi\Vert_{C^{\beta}_{b,d}}\Bigl(1+t^{-\frac{\gamma-\beta}{\alpha}}\Bigr), \quad t>0.
\end{equation}
The above estimate can be obtained through interpolation techniques (see Equation \eqref{eq:Interpolation_ineq}), once
sharp controls in supremum norm (Theorem \ref{prop:Control_in_Holder_norm} below) are established for the spatial derivatives of $P_t\phi$
when $\phi \in C^\beta_{b,d}(\R^N)$. We think that such an estimate \eqref{eq:Intro:Gradient_Estimates} and the controls in
Theorem \ref{prop:Control_in_Holder_norm} can be of independent interest and used also beyond our scope in other contexts. \newline
We face here two main difficulties to overcome. While in the gaussian setting, $L^\infty$-estimates of this type have been
established exploiting, for example, explicit formulas for the density of the semigroup $P_t$ (\cite{Lunardi97}), a priori controls of
Bernstein type combined with interpolation methods (\cite{Lorenzi05} and \cite{Saintier07}, when $n=2$ in \eqref{eq:def_of_n} below) or
probabilistic representations of the semigroup $P_t$, allowing Malliavin calculus (\cite{Priola09}), we cannot rely on
these techniques in our non-local framework, mainly due to the lower integrability properties for $P_t$. Instead,
we are going to use a \emph{perturbative approach} which consists in considering the L\'evy operator $\mathcal{L}$ as a perturbation, in a
suitable sense, of an $\alpha$-stable operator, at least for the associated small jumps. Indeed, we can "decompose" the operator
$\mathcal{L}$ in a smoother part, $\mathcal{L}^\alpha$, whose L\'evy measure is given by
\[ \mu(d\theta)\frac{\mathds{1}_{(0,r_0]}(r)}{r^{1+\alpha}}dr \]
and a remainder part. It is precisely condition [\textbf{ND}] that allows such a decomposition, since it ensures the positivity of the
L\'evy measure
\[d\nu-d\mu \frac{\mathds{1}_{[0,r_0]}}{r^{1+\alpha}}dr\]
associated with the remainder term. The main difference with the previous techniques in the diffusive setting is that we will work mainly on
the truncated $\alpha$-stable contribution $\mathcal{L}^\alpha$, being the remainder term only bounded. \newline
Following \cite{Schilling:Sztonyk:Wang}, we will establish that the Hartman-Winter condition holds, ensuring the
existence of a smooth density for the semigroup associated with $\mathcal{L}^\alpha$ and then, the required gradient estimates.
Indeed, assumption [\textbf{ND}] roughly states that the small jump
contributions of $\nu$, the ones responsible for the creation of a density, are controlled from below by an $\alpha$-stable measure, whose
absolute continuity is well-known in our framework. \newline
On the other hand, we will have to deal with the degeneracy of the operator $\mathcal{L}$, that acts non-degenerately, through the embedding
matrix $B$, only on a subspace of dimension $d$. It will be managed adapting the reasonings firstly appeared in \cite{Huang:Menozzi16}.
Namely, we will show that the semigroup associated with the Ornstein-Uhlenbeck operator $\mathcal{L}^{ou}$ coincides with a
non-degenerate one but "multiplied" by a time-dependent matrix that precisely takes into account the original degeneracy of the operator
(see definition of matrix $\mathbb{M}_t$ in Section $2.1$).

\paragraph{Main Operators Considered.} We conclude this introduction showing that assumption [\textbf{ND}] applies to a large class of L\'evy
operators on $\R^d$. As already pointed out in \cite{Schilling:Sztonyk:Wang}, it is satisfied by any L\'evy measure $\nu$ that can be
decomposed in polar coordinates as
\[\nu(C) \, = \, \int_{0}^{\infty}\int_{\mathbb{S}^{d-1}}\mathds{1}_{C}(r\theta)Q(r,\theta)\,\mu(d\theta) \frac{dr}{r^{1+\alpha}}, \quad C
\in \mathcal{B}(R^d_0),\]
for a finite, non-degenerate (in the sense of Equation \eqref{eq:non_deg_measure}), measure $\mu$ on $\mathbb{S}^{d-1}$ and a Borel
function $Q\colon(0,\infty)\times \mathbb{S}^{d-1}\to \R$ such
that there exists $r_0>0$ so that
\[ Q(r,\theta) \, \ge  \, c>0, \quad \text{a.e. in }[0,r_0]\times \mathbb{S}^{d-1}.\]
In particular, assumption [\textbf{ND}] holds for the following families of "stable-like" examples with $\alpha \in (0,2)$:
\begin{enumerate}
  \item Stable operator \cite{book:Sato99}:
  \[Q(r,\theta) \, = \, 1;\]
  \item Truncated stable operator with $r_0>0$ \cite{Kim:Song08}:
  \[Q(r,\theta) \, = \, \mathds{1}_{(0,r_0]}(r);\]
  \item Layered stable operator with $\beta$ in $(0,2)$ and $r_0>0$ \cite{Houdre:Kawai07}:
  \[Q(r,\theta) \, = \, \mathds{1}_{(0,r_0]}(r)+\mathds{1}_{(r_0,\infty)}(r)r^{\alpha-\beta};\]
  \item Tempered stable operator \cite{Rosinski09}:
  \[Q(\cdot,\theta) \text{ completely monotone, $Q(0,\theta)>0$ and $Q(\infty,\theta)=0$ a.e.\ in }S^{d-1};\]
  \item Relativistic stable operator \cite{Carmona:Masters:Simon90}, \cite{Byczkowski:Malecki:Ryznar09}:
  \[Q(r,\theta) \, = \, (1+r)^{(d+\alpha-1)/2}e^{-r};\]
  \item Lamperti stable operator with $f\colon S^{d-1}\to \R$ such that $\sup f(\theta)<1+\alpha$ \cite{Caballero:Pardo:Perez10}:
  \[Q(r,\theta) \, = \, e^{rf(\theta)}\Bigl(\frac{r}{e^r-1}\Bigr)^{1+\alpha}.\]
\end{enumerate}

\setcounter{equation}{0}
\paragraph{Organization of Paper.} The article is organized as follows. Section $2$ introduces some useful notations and then, the
anisotropic distance $d$ induced by the dynamics as well as Zygmund-H\"older spaces associated with such a distance. In Section $3$, we are
going to show some analytical properties of the semigroup $P_t$ generated by $\mathcal{L}^{ou}$, such as the existence of a
smooth density and, at least for small times, some controls for its derivatives. Section $4$ is then
dedicated to different estimates in the $L^\infty$-norm for $P_tf$ and its spatial derivatives, involving the supremum or the H\"older norm
of the function $f$. In particular, we show here the
continuity of $P_t$ as an operator between anisotropic Zygmund-H\"older spaces. In Section $5$, we use the controls established in the
previous parts in order to prove the elliptic Schauder estimates and show that Equation \eqref{eq:Elliptic_IPDE} has a unique weak solution.
Similarly, we establish the weak well-posedness of the Cauchy problem \eqref{eq:Parabolic_IPDE} as well as the associated parabolic
Schauder estimates. In the final section of the article, we briefly explain some possible extensions of the previous results to non-linear,
space-time dependent operators.

\section{Geometry of the Dynamics}
In this section, we are going to choose the right functional space "in which" to state our Schauder
estimates. The idea is to construct an H\"older space $C^\gamma_{b,d}(\R^N)$ with respect to a distance $d$ that it is
homogeneous to the dynamics, i.e. such that for $\gamma$ in $(0,1)$ and $f$ in $C^\beta_{b,d}(\R^N)$, any
distributional solution $u$ of
\begin{equation}\label{eq:distrib.sol}
\mathcal{L}^{\text{ou}}u(x) \, = \, \mathcal{L}u(x)+\langle Ax,Du(x)\rangle \, = \, f(x), \quad \text{$x$ in }\R^N
\end{equation}
is in $C^{\alpha+\beta}_{b,d}(\R^N)$, the expected parabolic bootstrap associated to this kind of operator. We recall in particular that the Kalman rank condition [\textbf{K}] is equivalent to the hypoellipticity (in the sense of H\"ormander \cite{Hormander67}) of the operator $\mathcal{L}^{\text{ou}}$ that
ensures the existence and smoothness of a distributional solution of Equation \eqref{eq:distrib.sol} for sufficiently regular $f$. See e.g.\
\cite{book:Ishikawa16} or \cite{Hao:Peng:Zhang19} for more details.

\subsection{The distance associated with the Dynamics}

To construct the suitable distance $d$, we start noticing that the Kalman rank condition [\textbf{K}] allows us to denote
\begin{equation}\label{eq:def_of_n}
n \,:= \, \min\{r \in \N \colon N \, = \, \rank \bigl[B,AB,\dots,A^{r-1}B\bigr]\}.
\end{equation}
Clearly, $n$ is in $\llbracket 1,N\rrbracket$, where $\llbracket \cdot, \cdot \rrbracket$ denotes the set of all the integers in the
interval, and $n=1$ if and only if $d=N$, i.e. if the dynamics is non-degenerate.\newline
As done in \cite{Lunardi97}, the space $\R^N$ will be decomposed with respect to the family of linear operators $B, AB,\dots,
A^{n-1}B$. We start defining the family $\{V_h\colon h\in \llbracket 1,n \rrbracket\}$ of subspaces of $\R^N$ through
\[
V_h \,  := \, \begin{cases}
            \text{Im} (B), & \mbox{if } h=1, \\
            \bigoplus_{k=1}^{h}\text{Im}(A^{k-1}B), & \mbox{otherwise}.
        \end{cases}\]
It is easy to notice that $V_h\neq V_k$ if $k\neq h$ and $V_1\subset V_2\subset\dots V_n=\R^N$. We can then construct iteratively the family
$\{E_h \colon h\in \llbracket 1,n \rrbracket\}$ of orthogonal projections from $\R^N$ as
\[E_h \,  := \,
        \begin{cases}
            \text{projection on } V_1, & \mbox{if } h=1; \\
            \text{projection on }(V_{h-1})^\perp \cap V_h, & \mbox{otherwise}.
        \end{cases}
\]
With a small abuse of notation, we will identify the projection operators $E_h$ with the corresponding matrixes in
$\R^N\otimes \R^N$. It is clear that $\dim E_1(\R^N)=d$. Let us then denote $d_1:=d$ and
\[d_h \,:= \, \dim E_h(\R^N), \quad \text{ for $h$}>1.\]

We can define now the distance $d$ through the decomposition $\R^N=\bigoplus_{h=1}^nE_h(\R^N)$ as
\[d(x,x') \,:= \, \sum_{h=1}^{n}\vert E_h(x-x')\vert^{\frac{1}{1+\alpha(h-1)}}.\]
The above distance can be seen as a generalization of the usual Euclidean distance when $n=1$ (non-degenerate dynamics) as well as an extension of the standard parabolic distance for $\alpha=2$. It is important to highlight that it does
not induce a norm since it lacks of linear homogeneity. \newline
The anisotropic distance $d$ can be understood direction-wise: we firstly fix a "direction" $h$ in $\llbracket
1,n\rrbracket$ and then calculate the standard Euclidean distance on the associated subspace $E_h(\R^N)$, but scaled according to the
dilation of the system in that direction. We conclude summing the contributions associated with each component. The choice of such a
dilation will be discussed thoroughly in the example at the end of this section.

As emphasized by the result from Lanconelli and Polidoro recalled below (cf. \cite{Lanconelli:Polidoro94}, Proposition $2.1$), the decomposition of
$\R^N$ with respect to the projections $\{E_h \colon h \in \llbracket 1,n\rrbracket\}$ determines a particular structure of the matrixes
$A$ and $B$. It will be often exploited in the following.

\begin{theorem}[\cite{Lanconelli:Polidoro94}]
\label{Thm:Lancon_Pol}
Let $\{e_i \colon i \in \llbracket 1,N\rrbracket\}$ be an orthonormal basis consisting of generators of $\{E_h(\R^N) \colon h\in \llbracket
1,n\rrbracket\}$. Then, the matrixes $A$ and $B$ have the following form:
\begin{equation}\label{eq:Lancon_Pol}
B \, = \,
    \begin{pmatrix}
        B_0    \\
        0      \\
        \vdots  \\
        0
    \end{pmatrix}
\,\, \text{ and } \,\, A \, = \,
    \begin{pmatrix}
        \ast   & \ast  & \dots  & \dots  & \ast   \\
         A_2   & \ast  & \ddots & \ddots  & \vdots   \\
        0      & A_3   & \ast  & \ddots & \vdots \\
        \vdots &\ddots & \ddots& \ddots & \ast \\
        0      & \dots & 0     & A_n    & \ast
    \end{pmatrix}
\end{equation}
where $B_0$ is a non-degenerate matrix in $\R^{d_1}\otimes \R^{d_1}$ and $A_h$ are matrixes in $\R^{d_h}\otimes \R^{d_{h-1}}$ with
$\text{rank}(A_h)=d_h$ for any $h$ in $\llbracket 2,n\rrbracket$. Moreover, $d_1\ge d_2\ge \dots\ge d_n\ge 1$.
\end{theorem}
Applying a change of variables if necessary, we will assume from this point further to have fixed such a canonical basis $\{e_i \colon i \in
\llbracket1,N\rrbracket\}$. For notational simplicity, we denote by $I_h$, $h \in \llbracket 1,n\rrbracket$, the family of indexes $i$ in
$\llbracket 1,N\rrbracket$ such that $\{e_i\colon i\in I_h\}$ spans $E_h(\R^N)$.

The particular structure of $A$ and $B$ given by Theorem \ref{Thm:Lancon_Pol} allows us to decompose accurately the
exponential $e^{tA}$ of the matrix $A$ in order to make the intrinsic scale of the system appear. Further on, we will consider fixed a time-dependent matrix $\mathbb{M}_t$ on $\R^N\otimes \R^N$ given by
\[\mathbb{M}_t := \text{diag}(I_{d_1\times d_1},tI_{d_2\times d_2},\dots,t^{n-1}I_{d_{n}\times d_{n}}), \quad t\ge0.\]

\begin{lemma}
\label{lemma:Decomposition_exp_A}
There exists a time-dependent matrix $\{R_t\colon t\in[0,1]\}$ in $\R^N\otimes \R^N$ such that
 \begin{equation}\label{eq:Decomposition_exp_A}
  e^{tA}\mathbb{M}_t \,   = \mathbb{M}_tR_t, \quad t \in [0,1].
 \end{equation}
Moreover, there exists a constant $C>0$ such that for any $t$ in $[0,1]$,
\begin{itemize}
  \item any $l,h$ in $\llbracket 1,n \rrbracket$ and any $\theta$ in $\mathbb{S}^{N-1}$, it holds that
  \[\bigl{\vert} E_le^{tA}E_h \theta\bigr{\vert} \, \le \, \begin{cases}
                                                      Ct^{l-h}, & \mbox{if } l\ge h \\
                                                      Ct, & \mbox{if } l< h.
                                                    \end{cases}\]
  \item any $\theta$ in $\mathbb{S}^{d-1}$, it holds that
  \[\bigl{\vert} R_tB\theta \bigr{\vert} \, \ge \, C^{-1}.\]
\end{itemize}
\end{lemma}
\begin{proof}
By definition of the matrix exponential, we know that
\begin{equation}\label{Proof:eq:def_Exponential}
E_le^{tA}E_h  \, = \, \sum_{k=0}^{\infty}\frac{t^k}{k!}E_lA^kE_h.
\end{equation}
Using now the representation of $A$ given by Theorem \ref{Thm:Lancon_Pol}, it is easy to check that $E_lA^kE_h=0$ for $k<l-h$  (when $l-h$
is non-negative). Thus, for $l\ge h$, it holds that
\[\bigl{\vert}E_le^{tA}E_h\theta\bigr{\vert}  \, = \, \bigl{\vert}\sum_{k=l-h}^{\infty}\frac{t^k}{k!}E_lA^kE_h\theta\bigr{\vert} \,
\le \, Ct^{l-h},\]
where we exploited that $t$ is in $ [0,1]$ and $\vert \theta \vert =1$. Assuming instead that $l<h$, it is clear that $E_lI_{N\times N}E_h$
vanishes. We can then write that
\[\bigl{\vert}E_le^{tA}E_h\theta\bigr{\vert}  \, = \, \bigl{\vert}\sum_{k=l}^{\infty}\frac{t^k}{k!}E_lA^kE_h\theta\bigr{\vert} \, \le \,
Ct,\]
using again that $t$ is in $ [0,1]$ and $\vert \theta \vert =1$.\newline
To show the other control, we highlight that the matrix $\mathbb{M}_t$ is not invertible in $t=0$ and for this reason, we
define the time-dependent matrix $R_t$ as
\[R_t \, := \, \begin{cases}
                 I_{N\times N}, & \mbox{if } t=0; \\
                 \mathbb{M}^{-1}_te^{tA}\mathbb{M}_t, & \mbox{if } t \in (0,1].
               \end{cases}\]
We could have also defined $R_t:=\bigr(\tilde{R}^t_s\bigr)_{|s=1}$ where $\tilde{R}^t_s$ solves the following ODE:
\[\begin{cases}
    \partial_s\tilde{R}^t_s \, = \, \mathbb{M}^{-1}_ttA\mathbb{M}_t\tilde{R}^t_s, \quad \text{on }(0,1], \\
    \tilde{R}^t_0 \, = \, I_{N\times N}.
  \end{cases}\]
Equivalently, $\tilde{R}^t_s$ is the resolvent matrix associated with $\mathbb{M}^{-1}_ttA\mathbb{M}_t$, whose sub-diagonal entries are
"macroscopic" from the structure of $A$ and $\mathbb{M}_t$. \newline
It follows immediately that Equation \eqref{eq:Decomposition_exp_A} holds.  Moreover, we notice that
\[\bigl{\vert} R_tB\theta\bigr{\vert} \, \ge \,\bigl{\vert} E_1R_tB\theta\bigr{\vert} \, = \, \bigl{\vert} E_1e^{tA}E_1B\theta\bigr{\vert}.\]
Remembering the definition of matrix exponential (Equation \eqref{Proof:eq:def_Exponential} with $l=h=1$), we use now that
\[E_1A^kE_1 \, = \, (E_1AE_1)^k \, = \, (A_{1,1})^kE_1,\]
where in the last expression the multiplication is meant block-wise, in order to conclude that
\[\bigl{\vert} R_t B\theta \bigr{\vert} \, \ge \,\bigl{\vert} e^{tA_{1,1}}B_0\theta\bigr{\vert}.\]
Using that $e^{tA_{1,1}}B_0$ is non-degenerate and continuous in time and that $\theta$ is in $\mathbb{S}^{d-1}$, it is easy to conclude.
\end{proof}

We conclude this sub-section with a simpler example taken from \cite{Huang:Menozzi:Priola19}. We hope that it will help the reader to
understand the introduction of the anisotropic distance $d$.
\begin{example}\label{example_basic}
Fixed $N=2d$, $n=2$ and $d=d_1=d_2$, we consider the following operator:
\[\mathcal{L}^{ou}_\alpha \, = \, \Delta^{\frac{\alpha}{2}}_{x_1}+x_1\cdot\nabla_{x_2} \quad \text{ on }\R^{2d},\]
where $(x_1,x_2)\in\R^{2d}$ and $\Delta^{\frac{\alpha}{2}}_{x_1}$ is the fractional Laplacian with respect to $x_1$. In our framework, it is associated with the matrixes
\[A \, := \,
    \begin{pmatrix}
               0 & 0 \\
               I_{d\times d} & 0
    \end{pmatrix} \,\, \text{ and } \,\, B \, := \,
    \begin{pmatrix}
                I_{d\times d} \\
                 0
    \end{pmatrix}.
             \]
The operator $\mathcal{L}^{ou}_\alpha$ can be seen as a generalization of the classical Kolmogorov example (see e.g. \cite{Kolmogorov34}) to
our non-local setting. \newline
In order to understand how the system typically behaves, we search for a dilation
\[\delta_\lambda\colon [0,\infty)\times \R^{2d} \to [0,\infty)\times \R^{2d}\]
which is invariant for the considered dynamics, i.e.\ a dilation that transforms solutions of the equation
\[\partial_tu(t,x)-\mathcal{L}^{ou}_\alpha u(t,x) \, = \, 0 \quad \text{ on }(0,\infty)\times\R^{2d}\]
into other solutions of the same equation.\newline
Due to the structure of $A$ and the $\alpha$-stability of $\Delta^{\frac{\alpha}{2}}$, we can consider for any fixed $\lambda>0$, the
following
\[ \delta_\lambda(t,x_1,x_2) := (\lambda^\alpha t,\lambda x_1,\lambda^{1+\alpha}x_2).\]
It then holds that
\[\bigl(\partial_t -\mathcal{L}^{ou}_\alpha\bigr) u = 0 \, \Longrightarrow \bigl(\partial_t -\mathcal{L}^{ou}_\alpha \bigr)(u \circ
\delta_\lambda) = 0.\]
Introducing now the complete time-space distance $d_P$ on $[0,\infty)\times \R^{2d}$ given by
\begin{equation}\label{Definition_distance_d_P}
d_P\bigl((t,x),(s,x')\bigr)  \, := \, \vert s-t\vert^\frac{1}{\alpha}+ d(x,x'),
\end{equation}
we notice that it is homogenous with respect to the dilation $\delta_\lambda$,
so that
\[d_P\bigl(\delta_\lambda(t,x);\delta_\lambda(s,x')\bigr) = \lambda d_P\bigl((t,x);(s,x')\bigr).\]
Precisely, the exponents appearing in Equation \eqref{Definition_distance_d_P} are those which make each space-component homogeneous to
the characteristic time scale $t^{1/\alpha}$. From a more probabilistic point of view, the exponents in Equation
\eqref{Definition_distance_d_P}, can be related to the characteristic time scales of the iterated integrals of an $\alpha$-stable process.
It can be easily seen from the example, noticing that the operator $\mathcal{L}^{ou}_\alpha$ corresponds to the generator of an isotropic $\alpha$-stable
process and its time integral. \newline
Going back to the general setting, the appearance of this kind of phenomena is due essentially to the particular structure of the
matrix $A$ (cf. Theorem \ref{Thm:Lancon_Pol}) that allows the smoothing effect of the operator $\mathcal{L}$, acting only on the first "component" given by $B_0$, to propagate into the system.
\end{example}

\subsection{Anisotropic Zygmund-H\"older spaces}
We are now ready to define the Zygmund-H\"older spaces $C^{\gamma}_{b,d}(\R^N)$ with respect to the distance $d$. We start recalling some useful
notations we will need below.\newline
Given a function $f\colon \R^N\to \R$, we denote by $Df(x)$, $D^2f(x)$ and $D^3f(x)$ the first, second and third Fr\'echet derivative of $f$
at a point $x$ in $\R^N$ respectively, when they exist. For simplicity, we will identify $D^3f(x)$ as a $3$-tensor so that $[D^3f(x)](u,v)$
is a vector in $R^N$ for any $u,v$ in $\R^N$. Moreover, fixed $h$ in $\llbracket 1, n\rrbracket$, we will denote by $D_{E_h}f(x)$ the
gradient of $f$ at $x$ along the direction $E_h(\R^N)$. Namely,
\[D_{E_h}f(x) \, \, := \, \, E_h Df(x).\]
A similar notation will be used for the higher derivatives, too.\newline
Given $X,Y$ two real Banach spaces, $\mathcal{L}(X,Y)$ will represent the family of linear continuous operators between $X$ and $Y$. \newline
In the following, $c$ or $C$ denote generic \emph{positive} constants whose precise value is unimportant. They may change from line
to line and they will depend only on the parameters given by the model and assumptions [\textbf{ND}], [\textbf{K}]. Namely,
$d,N,A,B,\alpha,\nu,r_0$ and $\mu$. Other dependencies that may occur will be explicitly specified.

Let us introduce now some function spaces we are going to use. We denote by $B_b(\R^N)$ the family of Borel measurable and bounded
functions $f\colon \R^N\to \R$. It is a Banach space endowed with the supremum norm $\Vert \cdot \Vert_\infty$. We will consider also its
closed subspace $C_b(\R^N)$ consisting of all the uniformly continuous functions. \newline
Fixed some $k$ in $\N_0:=\N\cup\{0\}$ and $\beta$ in $(0,1]$, we follow Lunardi \cite{Lunardi97} denoting the Zygmund-H\"older semi-norm for a function $\phi\colon \R^N\to \R$ as
\[[\phi]_{C^{k+\beta}} \, := \,
\begin{cases}
    \sup_{\vert \vartheta \vert= k}\sup_{x\neq y}\frac{\vert D^\vartheta\phi(x)-D^\vartheta\phi(y)\vert}{\vert x-y\vert^\beta} , & \mbox{if
    }\beta \neq 1; \\
    \sup_{\vert \vartheta \vert= k}\sup_{x\neq y}\frac{\bigl{\vert}D^\vartheta\phi(x)+D^\vartheta\phi(y)-2D^\vartheta\phi(\frac{x+y}{2})
    \bigr{\vert}}{\vert x-y \vert}, & \mbox{if } \beta =1.
                             \end{cases}\]
Consequently, The Zygmund-H\"older space $C^{k+\beta}_b(\R^N)$ is the family of functions $\phi\colon \R^N
\to\R$ such that $\phi$ and its derivatives up to order $k$ are continuous and the norm
\[\Vert \phi \Vert_{C^{k+\beta}_b} \,:=\, \sum_{i=1}^{k}\sup_{\vert\vartheta\vert = i}\Vert D^\vartheta\phi
\Vert_{L^\infty}+[\phi]_{C^{k+\beta}_b} \,\text{ is finite.}\]

We can define now the anisotropic Zygmund-H\"older spaces associated with the distance $d$. Fixed $\gamma>0$, the space $C^{\gamma}_{b,d}(\R^N)$ is
the family of functions $\phi\colon \R^N\to \R$ such that for any $h$ in $\llbracket 1,n\rrbracket$ and any $x_0$ in $\R^N$, the function
\[z\in  E_h(\R^N)\,  \to \, \phi(x_0+z) \in \R \,\text{ belongs to }C^{\gamma/(1+\alpha(h-1))}_b\bigl(E_h(\R^N)\bigr),\]
with a norm bounded by a constant independent from $x_0$. It is endowed with the norm
\begin{equation}\label{eq:def_anistotropic_norm}
\Vert\phi\Vert_{C^{\gamma}_{b,d}} \,:=\,\sum_{h=1}^{n}\sup_{x_0\in \R^N}\Vert\phi(x_0+\cdot)\Vert_{C^{\gamma/(1+\alpha(h-1))}_b}.
\end{equation}

We highlight that it is possible to recover the expected joint regularity for the partial derivatives, when
they exist, as in the standard H\"older spaces. In such a case, they actually turn out to be H\"older continuous with respect to the
distance $d$ with order one less than the function (See Lemma $2.1$ in \cite{Lunardi97} for more details).

It will be convenient in the following to consider an equivalent norm in the "standard" H\"older-Zygmund spaces
$C^\gamma_b(E_h(\R^N))$ that does not take into account the derivatives with respect to the different directions.
We suggest the interested reader to see \cite{Lunardi97}, Equation $(2.2)$ or \cite{Priola09} Lemma $2.1$ for further details.

\begin{lemma}
\label{lemma:def_equivalent_norm}
Fixed $\gamma$ in $(0,3)$ and $h$ in $\llbracket 1,n\rrbracket$ and $\phi$ in $C_b(E_h(\R^N))$, let us introduce
\begin{equation}\label{eq:def_equivalent_norm}
\Delta^3_{x_0}\phi(z) \, := \, \phi(x_0+3z)-3\phi(x_0+2z)+ 3\phi(x_0+z)-\phi(x_0), \,\quad x_0 \in \R^N; \, z \in E_h(\R^N).
\end{equation}
 Then, $\phi$ is in $C^\gamma_{b}(E_h(\R^N))$ if and only if
\[\sup_{x_0 \in \R^N}\sup_{z \in E_h(\R^N);z\neq 0}\frac{\bigl{\vert} \Delta^3_{x_0}\phi(z)\bigr{\vert}}{\vert z \vert^\gamma} \, < \,
\infty.\]
\end{lemma}

We conclude this subsection with a result concerning the interpolation between the anisotropic
Zygmund-H\"older spaces $C^{\gamma}_{b,d}(\R^N)$. We refer to  Theorem $2.2$ and Corollary $2.3$ in \cite{Lunardi97} for details.

\begin{theorem}\label{Thm:Interpolation}
Let $r$ be in $(0,1)$ and $\beta$, $\gamma$ in $[0,\infty)$ such that $\beta\le \gamma$. Then, it holds that
\[\bigl(C^\beta_{b,d}(\R^N),C^\gamma_{b,d}(\R^N)\bigr)_{r,\infty} \, = \, C^{r\gamma+(1-r)\beta}_{b,d}(\R^N)\]
with equivalent norms, where we have denoted for simplicity: $C^0_{b,d}(\R^N):=C_b(\R^N)$.
\end{theorem}

\setcounter{equation}{0}
\section{Smoothing Effect for Truncated Density}

We present here some analytical properties of the semigroup generated by the operator $\mathcal{L}^{\text{ou}}$. Following
 \cite{Schilling:Sztonyk:Wang} and \cite{Schilling:Wang12}, we will show the existence of a smooth density for such a
semigroup and its anisotropic smoothing effect, at least for small times.

Throughout this section, we consider fixed a stochastic base $\bigl(\Omega,\mathcal{F},(\mathcal{F}_t)_{t\ge 0},\mathbb{P}\bigr)$ satisfying
the usual assumptions (see
\cite{book:Applebaum09}, page $72$). Let us then consider the (unique in law) L\'evy process $\{Z_t\}_{t\ge 0}$ on $\R^d$ characterized by
the L\'evy symbol
\[\Phi(p) \, = \, -ib\cdot p +\frac{1}{2}p\cdot Qp + \int_{\R^d_0}\bigl(1-e^{ip \cdot z}+ip\cdot z\mathds{1}_{B(0,1)}(z)\bigr) \, \nu(dz),
\quad p \in \R^d.\]
It is well-known by the L\'evy-Kitchine formula (see
\cite{book:Jacob05}), that the infinitesimal generator of the process $\{BZ_t\}_{t\ge 0}$ is then given by $\mathcal{L}$ on $\R^N$.\newline
Fixed $x$ in $\R^N$, we denote by $\{X_t\}_{t\ge 0}$ the $N$-dimensional Ornstein-Uhlenbeck process driven by $BZ_t$, i.e.\ the unique
(strong) solution of the following stochastic differential equation:
\[X_t \, = \, x + \int_{0}^{t}AX_s \, ds +BZ_t, \quad t\ge0, \,\, \mathbb{P}\text{-almost surely.}\]
By the variation of constants method, it is easy to check that
\begin{equation}\label{eq:Def_OU_Process}
X_t \, = \, e^{tA}x + \int_{0}^{t}e^{(t-s)A}B\, dZ_s, \quad \quad t\ge0, \,\, \mathbb{P}\text{-almost surely.}
\end{equation}
The \emph{transition semigroup} associated with $\mathcal{L}^{\text{ou}}$ is then defined as the family
$\{P_t\colon t \ge 0\}$ of linear contractions on $B_b(\R^N)$ given by
\begin{equation}\label{eq:Def_transition_semigroup}
P_t\phi(x) \, = \, \mathbb{E}\bigl[\phi(X_t)\bigr], \quad x \in \R^N,\, \phi \in B_b(\R^N).
\end{equation}
We recall that $P_t$ is generated by $\mathcal{L}^{\text{ou}}$ in the sense that its infinitesimal generator
$\mathcal{A}$ coincides with $\mathcal{L}^{\text{ou}}$ on $C^\infty_c(\R^N)$, the family of smooth functions with compact support.

The next result shows that the random part of $X_t$ (see Equation \eqref{eq:def_Lambda_t}) satisfies again the non-degeneracy assumption [\textbf{ND}], even if re-scaled
with respect to the anisotropic structure of the dynamics.

\begin{prop}[Decomposition]
\label{prop:Decomposition_Process_X}
For any $t$ in $(0,1]$, there exists a L\'evy process $\{S^t_u\}_{u\ge 0}$ such that
\[X_t \, \overset{law}{=} \, e^{tA}x +\mathbb{M}_t S^t_t.\]
Moreover, $\{S^t_u\}_{u\ge 0}$ satisfies assumption [\textbf{ND}] with same $\alpha$ as before.
\end{prop}
\begin{proof}
For simplicity, we start denoting
\begin{equation}\label{eq:def_Lambda_t}
\Lambda_t \,:= \, \int_{0}^{t}e^{(t-s)A}B\, dZ_s,\quad t> 0,
\end{equation}
so that $X_t=e^{tA}x+\Lambda_t$. To conclude, we need to construct a L\'evy process $\{S^t_u\}_{u\ge 0}$ on $\R^N$ satisfying assumption [\textbf{ND}] and
\begin{equation}\label{eq:identity_in_law}
\Lambda_t \, \overset{law}{=} \, \mathbb{M}_tS_t^t.
\end{equation}
To show the identity in law, we are going to reason in terms of the characteristic functions. By Lemma $2.2$ in \cite{Schilling:Wang12}, we
know that $\Lambda_t$ is an infinitely divisible random variable with associated L\'evy symbol
\[\Phi_{\Lambda_t}(\xi) \, := \, \int_{0}^{t}\Phi\bigl((e^{sA}B)^*\xi\bigr)\,ds, \quad \xi \in \R^N.\]
Remembering the decomposition $e^{sA}B=e^{sA}\mathbb{M}_sB=\mathbb{M}_{s}R_{s}B$ from Lemma \ref{lemma:Decomposition_exp_A}, we can now
rewrite $\Phi_{\Lambda_t}$ as
\[\Phi_{\Lambda_t}(\xi) \, = \, t\int_{0}^{1}\Phi\bigl((e^{stA}B)^*\xi\bigr)\,ds\, = \,
t\int_{0}^{1}\Phi\bigl((R_{st}B)^*\mathbb{M}_s\mathbb{M}_t\xi\bigr)\,ds.\]
The above equality suggests us to define, for any fixed $t$ in $(0,1]$, the (unique in law) L\'evy process  $\{S^t_u\}_{u\ge0}$ associated with the L\'evy symbol
\[\tilde{\Phi}^t(\xi) \, := \, \int_{0}^{1}\Phi\bigl((R_{st}B)^*\mathbb{M}_s\xi\bigr)\,ds, \quad \xi \in \R^N.\]
It is not difficult to check that $\tilde{\Phi}^t$ is indeed a L\'evy symbol associated with the L\'evy triplet
$(\tilde{Q}^t,\tilde{b}^t,\tilde{\nu}^t)$ given by
\begin{align}
  \tilde{Q}^t \, &= \, \int_{0}^{1}\mathbb{M}_sR_{st}BQ(\mathbb{M}_sR_{st}B)^\ast \, ds; \\
  \tilde{b}^t \, = \, \int_{0}^{1}\mathbb{M}_sR_{st}Bb \, ds &+ \int_{0}^{1}\int_{\R^d}
  \mathbb{M}_sR_{st}Bz\bigl[\mathds{1}_{B(0,1)}(\mathbb{M}_sR_{st}Bz)-\mathds{1}_{B(0,1)}(z)\bigr]
  \, \nu(dz)ds;\\
  \tilde{\nu}^t(C) \, &= \, \int_{0}^{1} \nu\bigl((\mathbb{M}_sR_{st}B)^{-1}C\bigr)\, ds, \quad C \in \mathcal{B}(\R^d_0).
\end{align}
 Since we have that
\[\mathbb{E}\bigl[e^{i\langle \xi,\Lambda_t\rangle}\bigr] \, = \, e^{-\Phi_{\Lambda_t}(\xi)}\, = \, e^{-t\tilde{\Phi}^t(\mathbb{M}_t\xi)} \,
= \,\mathbb{E}\bigl[e^{i\langle \xi,\mathbb{M}_tS^t_t\rangle}\bigr],\]
it follows immediately that the identity \eqref{eq:identity_in_law} holds.\newline
It remains to show that the family of L\'evy measure $\{\tilde{\nu}^t\colon t \in (0,1]\}$ satisfies the assumption [\textbf{ND}].
Recalling that condition [\textbf{ND}] is assumed to hold for $\nu$, we know that
\begin{equation}\label{Proof:eq:Decomposition}
\tilde{\nu}^t(C) \, = \, \int_{0}^{1} \nu\bigl((\mathbb{M}_sR_{st}B)^{-1}C\bigr)\, ds \, \ge \,
\int_{0}^{1}\int_{0}^{r_0}\int_{\mathbb{S}^{d-1}} \mathds{1}_{C}(r\mathbb{M}_sR_{st}B\theta) \mu(d\theta)\frac{dr}{r^{1+\alpha}}ds,
\end{equation}
for any $C$ in $\mathcal{B}(\R^d_0)$. Furthermore, it holds from Lemma \ref{lemma:Decomposition_exp_A} that
\begin{equation}\label{Proof:eq:Decomposition2}
\inf_{s\in (0,1),\, t \in (0,1],\, \theta \in S^{d-1}}\bigl{\vert} \mathbb{M}_sR_{st}B\theta \bigr{\vert}\, =:\, R_0>0.
\end{equation}
It allows us to define two functions $l^t\colon [0,1]\times
S^{d-1}\to S^{N-1}$, $m^t\colon [0,1]\times S^{d-1}\to \R$, given by
\[l^t(s,\theta) \,:= \, \frac{\mathbb{M}_sR_{st}B\theta}{\vert \mathbb{M}_sR_{st}B\theta\vert} \,\, \text{ and } \,\,  m^t(s,\theta) \,:= \,
\vert \mathbb{M}_sR_{st}B\theta\vert.\]
Using the Fubini theorem, we can now rewrite Equation \eqref{Proof:eq:Decomposition} as
\[
\begin{split}
\tilde{\nu}^t(C) \, &\ge \,\int_{0}^{1}\int_{\mathbb{S}^{d-1}}\int_{0}^{r_0}\mathds{1}_{C}(l^t(s,\theta)m^t(s,\theta)r)
\frac{dr}{r^{1+\alpha}}\mu(d\theta)ds \\
&= \, \int_{0}^{1}\int_{\mathbb{S}^{d-1}} \int_{0}^{m^t(s,\theta)r_0}\mathds{1}_{C}(l^t(s,\theta)r)
\frac{dr}{r^{1+\alpha}}[m^t(s,\theta)]^\alpha\mu(d\theta)ds.
\end{split}
\]
Exploiting again Control \eqref{Proof:eq:Decomposition2}, we can conclude that
\begin{equation}\label{eq:ND_assumption_for_St}
\tilde{\nu}^t(C) \, \ge \, \int_{0}^{1}\int_{\mathbb{S}^{d-1}} \int_{0}^{R_0}\mathds{1}_{C}(l^t(s,\theta)r)
\frac{dr}{r^{1+\alpha}}\tilde{m}^t(ds,d\theta) \, = \, \int_{0}^{R_0}\int_{\mathbb{S}^{N-1}}
\mathds{1}_{C}(\tilde{\theta} r) \tilde{\mu}^t(d\tilde{\theta})\frac{dr}{r^{1+\alpha}},
\end{equation}
where $\tilde{m}^t(ds,d\theta)$ is a measure on $[0,1]\times S^{d-1}$ given by
\[\tilde{m}^t(ds,d\theta) \, := \, [m^t(s,\theta)]^{\alpha}\mu(d\theta)ds\]
and $\tilde{\mu}^t:=(l^t)_{\ast}\tilde{m}^t$ is the measure $\tilde{m}^t$ push-forwarded through $l^t$ on $S^{N-1}$.
It is easy to check that the measure $\tilde{\mu}^t$  is finite and non-degenerate in the sense of \eqref{eq:non_deg_measure},
replacing therein $d$ by $N$.
\end{proof}

An immediate application of the above result is a first representation formula for the transition semigroup $\{P_t\colon t\ge0\}$ associated
with the Ornstein-Uhlenbeck process $\{X_t\}_{t\ge 0}$, at least for small times. Indeed, denoting by $\mathbb{P}_{X}$ the law of a
random variable $X$, Equation \eqref{eq:identity_in_law} implies that for any $\phi$ in $B_b(\R^N)$, it holds that
\begin{equation}\label{eq:Representation_semigroup1}
P_t\phi(x) \, = \, \int_{\R^N}\phi(e^{tA}x+y) \, \mathbb{P}_{\Lambda_t}(dy) \, = \, \int_{\R^N}\phi(e^{tA}x+\mathbb{M}_ty) \,
\mathbb{P}_{S^t_t}(dy), \quad x \in \R^N, \, t \in (0,1].
\end{equation}
Moreover, condition [\textbf{ND}] for $\{S^t_u\}_{u\ge 0}$ allows us to decompose it into two components: a truncated, $\alpha$-stable part
and a remainder one. Indeed, if we denote by $\nu^t_{\alpha}$ the measure serving as lower bound to the L\'evy measure $\tilde{\nu}^t$ in \eqref{eq:ND_assumption_for_St}, i.e.
\begin{equation}\label{eq:def_levy_measure_for_stable}
\nu^t_{\alpha}(C) \, := \, \int_{0}^{R_0}\int_{\mathbb{S}^{N-1}} \mathds{1}_{C}(\theta r) \tilde{\mu}^t(d\theta)\frac{dr}{r^{1+\alpha}},
\quad C \in \mathcal{B}(\R^N_0),
\end{equation}
we can consider $\{Y^t_u\}_{u\ge 0}$, the L\'evy process on $\R^N$ associated with the L\'evy triplet $(0,0, \nu^t_{\alpha})$. We recall now a useful fact involving the L\'evy symbol $\Phi^t_\alpha$ of the process $Y^t$. The non-degeneracy of the
measure $\tilde{\mu}^t$ is equivalent to the existence of a constant $C>0$ such that
\begin{equation}\label{eq:equivalence_non-deg_measure}
\Phi^t_\alpha(\xi) \, \ge \, C \vert \xi \vert^\alpha, \quad \xi \in \R^N.
\end{equation}
A proof of this result can be found, for example, in \cite{Priola12} p.$424$.\newline
In order to apply the results in \cite{Schilling:Sztonyk:Wang}, we are going to truncate the above process at the typical
time scale for an $\alpha$-stable process. This is $t^{1/\alpha}$ when considering the process at time $t$ (cf. Example \ref{example_basic}). Namely, we consider the family $\{\mathbb{P}^{\text{tr}}_t\}_{t\ge 0}$ of infinitely
divisible probabilities whose characteristic function has the form
$\widehat{\mathbb{P}^{\text{tr}}_t}(\xi):=\text{exp}[-\Phi^{\text{tr}}_t(\xi)]$, where
\[\Phi^{\text{tr}}_t(\xi) \, := \, \int_{\vert z\vert \le t^{\frac{1}{\alpha}}} \bigl[1-e^{i\langle \xi,z\rangle}+i\langle
\xi,z\rangle\bigr]\, \nu^t_\alpha(dz).\]
On the other hand, since the measure $\tilde{\nu}^t$ satisfies assumption [\textbf{ND}], we know that the remainder
$\tilde{\nu}^t-\mathds{1}_{B(0,t^{1/\alpha})}\nu^t_\alpha$ is again a L\'evy measure on $\R^N$. Let $\{\pi_t\}_{t\ge0}$ be the family of
infinitely divisible probability associated with the following L\'evy triplet:
\[(\tilde{Q}^t,\tilde{b}^t,\tilde{\nu}^t-\mathds{1}_{B(0,t^{1/\alpha})}\nu^t_\alpha).\]
It follows immediately that $\mathbb{P}_{S^t_t} \, = \, \mathbb{P}^{\text{tr}}_t \ast \pi_t$ for any $t>0$. We can now disintegrate the
measure $\mathbb{P}_{S^t_t}$ in Equation \eqref{eq:Representation_semigroup1} in order to obtain
\begin{equation}\label{eq:Representation_semigroup2}
P_t\phi(x) \, = \, \int_{\R^N}\int_{\R^N}\phi\bigl(e^{tA}x+\mathbb{M}_t(y_1+y_2)\bigr) \, \mathbb{P}^{\text{tr}}_t(dy_1) \pi_t(dy_2).
\end{equation}

The next step is to use Proposition $2.3$ in \cite{Schilling:Sztonyk:Wang} to show a smoothing effect for the family of truncated
stable measures $\{\mathbb{P}^{\text{tr}}_t\colon t\ge 0\}$, at least for small times. Namely,

\begin{prop}
\label{prop:control_on_truncated_density}
Fixed $m$ in $\N_0$, there exists $T_0:=T_0(m)>0$ such that for any $t$ in $(0,T_0]$, the
probability $\mathbb{P}^{\text{tr}}_{t}$ has a density $p^{\text{tr}}(t,\cdot)$ that is $m$-times continuously differentiable on $\R^N$. \newline
Moreover, for any $\vartheta$ in $\mathbb{N}^N$ such that $\vert \vartheta \vert \le m$, there exists a constant $C:=C(m,\vert
\vartheta\vert)$ such that
\[\vert D^\vartheta p^{\text{tr}}(t,y) \vert \, \le \, Ct^{-\frac{N+\vert \vartheta \vert}{\alpha}} \Bigl(1+\frac{\vert y
\vert}{t^{1/\alpha}}\Bigr)^{\vert \vartheta \vert-m}, \quad t \in (0,T_0], \, y \in \R^N.\]
\end{prop}
\begin{proof}
The result follows immediately applying Proposition $2.3$ in \cite{Schilling:Sztonyk:Wang}. To do so, we need to show that the L\'evy symbol
$\Phi^t_\alpha$ of the process $\{Y^t_u\}_{u\ge 0}$ satisfies the following assumptions:
\begin{itemize}
  \item \emph{Hartman-Wintner condition.} There exists $T>0$ such that
  \[\liminf_{\vert \xi \vert \to \infty}\frac{\text{Re}\Phi^t_{\alpha}(\xi)}{\ln(1+\vert \xi \vert)}\, = \, \infty, \quad t \in (0,T];\]
  \item \emph{Controllability condition.} There exist $T>0$ and $c>0$ such that
  \[\int_{\R^N}e^{-t\text{Re}\Phi^t_{\alpha}(\xi)}\vert \xi \vert^m \, \le \, ct^{-\frac{m+N}{\alpha}}, \quad t \in (0,T].\]
\end{itemize}
In order to show that the above conditions hold, we fix $T\le 1$ and we recall that the L\'evy symbol $\Phi^t_\alpha$ of $Y_t$, the
truncated $\alpha$-stable process with L\'evy measure introduced in \eqref{eq:def_levy_measure_for_stable}, can be written through the
L\'evy-Kitchine formula as
\[\Phi^t_\alpha(\xi) \, = \, \int_{\R^N_0}\bigl(1-e^{i\langle \xi,z\rangle}+i\langle \xi,z\rangle\bigr)\nu_\alpha^t(dz) \, = \,
\int_{0}^{R_0}\int_{\mathbb{S}^{N-1}}\bigl(1-\cos(\langle \xi, r\theta\rangle)\bigr)\, \tilde{\mu}^t(d\theta)\frac{dr}{r^{1+\alpha}}.\]
We have seen in Equation \eqref{eq:equivalence_non-deg_measure} that the non-degeneracy of $\tilde{\mu}^t$ implies that
$\Phi^t_\alpha(\xi)  \ge C \vert \xi \vert^\alpha$. The Hartman-Wintner condition then follows immediately since
\[\liminf_{\vert \xi \vert \to \infty}\frac{\text{Re}\Phi^t_{\alpha}(\xi)}{\ln(1+\vert \xi \vert)}\, \ge \, \liminf_{\vert \xi \vert \to
\infty}\frac{c\vert\xi\vert^\alpha}{\ln(1+\vert \xi \vert)}\, = \,\infty.\]
To show instead the controllability assumption, let us firstly notice that
\[e^{-t\text{Re}\Phi^t_{\alpha}(\xi)} \, \le \,
            \begin{cases}
                1, & \mbox{if } \vert \xi \vert \le R; \\
                e^{-ct\vert \xi \vert^\alpha}, & \mbox{if } \vert \xi \vert > R,
            \end{cases}
\]
for some $R>0$. It then follows that
\begin{align*}
 \int_{\R^N} e^{-t\text{Re}\Phi^t_{\alpha}(\xi)}\vert \xi \vert^m \, d\xi \, &= \, \int_{\vert \xi \vert\le R}\vert \xi \vert^m \, d\xi+
 \int_{\vert \xi \vert>R} e^{-ct\vert\xi\vert^\alpha}\vert \xi \vert^m \, d\xi \\
   &\le \, C+t^{-\frac{m+N}{\alpha}}\int_{\vert \xi \vert>t^{1/\alpha}R}e^{-c\vert\xi\vert^\alpha}\vert \xi \vert^m \, d\xi\\
   &\le \, C+t^{-\frac{m+N}{\alpha}}\int_{\R^N}e^{-c\vert\xi\vert^\alpha}\vert \xi \vert^m \, d\xi \\
   &\le \, Ct^{-\frac{m+N}{\alpha}},
\end{align*}
where in the last step we used that $1\le t^{-\frac{m+N}{\alpha}}$.
\end{proof}

\setcounter{equation}{0}
\section{Estimates for Transition Semigroup}

The results in the previous section (Proposition \ref{prop:control_on_truncated_density} and Equation \eqref{eq:Representation_semigroup2})
allow us to represent the semigroup $P_t$ of the Ornstein-Uhlenbeck process $\{X_t\}_{t\ge 0}$ as
\begin{equation}\label{eq:Representation_semigroup}
P_t\phi(x) \, = \, \int_{\R^N}\int_{\R^N}\phi(\mathbb{M}_t(y_1+y_2)+e^{tA}x)p^{\text{tr}}(t,y_1)\, dy_1\pi_t(dy_2), \quad x \in \R^N,
\end{equation}
at least for small time intervals. \newline
Here, we will focus on estimates in $\Vert \cdot\Vert_\infty$-norm of the transition semigroup $\{P_t\colon t\ge 0\}$  given in Equation
\eqref{eq:Def_transition_semigroup} and its derivatives. The main result in this section is Corollary
\ref{corollary:continuity_between_holder} that shows the continuity of $P_t$ between anisotropic Zygmund-H\"older spaces. These controls
will be fundamental in the next section to prove Schauder Estimates in the elliptic and parabolic settings.\newline
As we will see in the following result, the derivatives of the semigroup $P_t$ with respect to a component $i$ in $I_h$ induces an additional time singularity of order $\frac{1+\alpha(h-1)}{\alpha}$, corresponding to the intrinsic time scale of the considered component.

\begin{prop}
\label{prop:control_in_inf_norm}
Let $h,h',{h''}$ be in $\llbracket 1,n\rrbracket$ and $\phi$ in $B_b(\R^N)$. Then, there exists a constant $C>0$ such that for any $i$ in
$I_h$, any $j$ in $I_{h'}$ and any $k$ in $I_{h''}$, it holds that
\begin{align}
\label{eq:Control_in_inf_norm_deriv1}
    \Vert D_i P_t \phi \Vert_\infty \, &\le \, C\Vert \phi \Vert_\infty \bigl(1+t^{-\frac{1+\alpha(h-1)}{\alpha}}\bigr), \quad t>0;\\
\label{eq:Control_in_inf_norm_deriv2}
    \Vert D^2_{i,j} P_t \phi \Vert_\infty \, &\le \, C\Vert \phi \Vert_\infty \bigl(1+t^{-\frac{2+\alpha(h+h'-2)}{\alpha}}\bigr),\quad t>0;\\
\label{eq:Control_in_inf_norm_deriv3}
  \Vert D^3_{i,j,k} P_t \phi \Vert_\infty \, &\le \, C\Vert \phi \Vert_\infty \bigl(1+t^{-\frac{3+\alpha(h+h'+h''-3)}{\alpha}}\bigr),\quad
  t>0.
\end{align}
\end{prop}
\begin{proof}
We start fixing a time horizon  $T := 1\wedge  T_0(N+4) > 0$, where $T_0(m)$ was defined in Proposition
\ref{prop:control_on_truncated_density}. Our choice of $N+4$ is motivated by the fact that we consider derivatives up to order $3$. \newline
On the interval $(0,T]$, the representation formula \eqref{eq:Representation_semigroup} holds and $P_t\phi$ is three times differentiable for any $\phi$ in $B_b(\R^N)$. We are going to show only Estimate \eqref{eq:Control_in_inf_norm_deriv1} since the controls for the higher derivatives can be obtained similarly.\newline
Fixed $t\le T$, let us consider $i$ in $I_h$ for some $h$ in $\llbracket 1,n\rrbracket$. When $t\le T$, we recall from Equation
\eqref{eq:Representation_semigroup} that, up to a change of variables, it holds that
\[\bigl{\vert}D_i P_t\phi(x)\bigr{\vert} \, = \, \Bigl{\vert} D_i\int_{\R^N}\int_{\R^N} \phi(\mathbb{M}_t(y_1+y_2)) p^{\text{tr}}(t,
y_1-\mathbb{M}^{-1}_te^{tA}x)\, dy_1\pi_t(dy_2) \Bigr{\vert}.\]
We can then move the derivative inside the integral and write that
\begin{equation}\label{eq:Control_deriv1}
\begin{split}
\bigl{\vert}D_i P_t\phi(x)\bigr{\vert} \, &= \, \Bigl{\vert} \int_{\R^N}\int_{\R^N} \phi(\mathbb{M}_t(y_1+y_2)) \langle \nabla p^{\text{tr}} (t,
y_1-\mathbb{M}^{-1}_te^{tA}x),\mathbb{M}^{-1}_te^{tA}e_i\rangle\, dy_1\pi_t(dy_2)\Bigr{\vert} \\
&\le \, \vert \mathbb{M}^{-1}_te^{tA} e_i \vert \int_{\R^N}\int_{\R^N} \vert \phi(\mathbb{M}_t(y_1+y_2))\vert \, \vert \nabla p^{\text{tr}}(t,y_1-\mathbb{M}^{-1}_te^{tA}x))\vert\, dy_1\pi_t(dy_2)\\
&\le \, Ct^{-(h-1)}\Vert \phi\Vert_\infty \int_{\R^N}\int_{\R^N} \vert \nabla p^{\text{tr}} (t, y_1)\vert \, dy_1\pi_t(dy_2),
\end{split}
\end{equation}
where in the last step we exploited Lemma \ref{lemma:Decomposition_exp_A} to control
\[\vert \mathbb{M}^{-1}_te^{tA}e_i\vert \, \le \, \sum_{k=1}^{n}\vert  \mathbb{M}^{-1}_tE_ke^{tA}E_he_i\vert \, \le \,
C\bigl[\sum_{k=1}^{h-1}t^{k-h}t+\sum_{k=h}^{n}t^{-(k-1)}t^{-(h-1)}\bigr] \, \le \, Ct^{-(h-1)},\]
remembering that $t\le 1$.
We conclude the case $t\le T$ using the control on $p^{\text{tr}}$ (Proposition \ref{prop:control_on_truncated_density} with $m=N+2$) to
write that
\begin{equation}\label{eq:Control_deriv3}
\begin{split}
\bigl{\vert}D_i P_t\phi(x)\bigr{\vert}\, &\le \,  C\Vert \phi \Vert_\infty\pi_t(\R^N)t^{-(h-1)}  \int_{\R^N}
t^{-\frac{N+1}{\alpha}}\bigl(1+\frac{\vert y_1 \vert}{t^{1/\alpha}}\bigr)^{-(N+1)} \, dy_1 \\
&\le \, C\Vert \phi \Vert_\infty t^{-\frac{1+\alpha(h-1)}{\alpha}}  \int_{\R^N} (1+\vert z \vert)^{-(N+1)}\, dz\\
&\le \,C\Vert \phi \Vert_\infty t^{-\frac{1+\alpha(h-1)}{\alpha}}.
\end{split}
\end{equation}
Above, we used the change of variables $z=t^{-1/\alpha}y_1$. When $t>T$, we can exploit the already proven controls for small times, the
semigroup and the contraction properties of $\{P_t\colon t\ge0\}$
on $B_b(\R^N)$ to write that
\begin{equation}\label{eq:COntrol_for_t>1}
\Vert D_i P_t \phi \Vert_\infty \, = \, \Vert D_i P_{T}\bigl(P_{t-T} \phi\bigr) \Vert_\infty \, \le \, C_T\Vert P_{t-T} \phi
\Vert_\infty \, \le \, C\Vert \phi \Vert_\infty.
\end{equation}
We have thus shown Control \eqref{eq:Control_in_inf_norm_deriv1} for any $t>0$.
\end{proof}

The following interpolation inequality (see e.g.\ \cite{book:Triebel83})
\begin{equation}\label{eq:Interpolation_ineq}
\Vert \phi \Vert_{C^{r\delta_1+(1-r)\delta_2}_b} \, \le \, C\Vert \phi\Vert^r_{C^{\delta_1}_b}\Vert \phi \Vert^{1-r}_{C^{\delta_2}_b}
\end{equation}
valid for $0\le \delta_1 <\delta_2$, $r$ in $(0,1)$ and $\phi$ in $C^{\delta_2}(\R^N)$, allows us to extend easily the above
result.

\begin{corollary}
\label{coroll:C0_Cgamma_control}
Let $\gamma$ be in $[0,1+\alpha)$. Then, there exists a constant $C>0$ such that
\begin{equation}\label{eq:coroll:C0_Cgamma_control}
\Vert P_t \Vert_{\mathcal{L}_c(C_b,C^\gamma_{b,d})} \, \le \, C\bigl(1+t^{-\frac{\gamma}{\alpha}}\bigr), \quad t>0.
\end{equation}
\end{corollary}
\begin{proof}
Let us firstly assume that $\gamma$ is in $(0,1]$. Remembering the definition of $C^\gamma_{b,d}$-norm in \eqref{eq:def_anistotropic_norm}, we start
fixing a point $x_0$ in $\R^N$ and $h$ in $\llbracket 2,n\rrbracket$. Then, the contraction property of the semigroup implies that
\[\Vert P_t\phi(x_0+\cdot)_{|E_h(\R^N)}\Vert_\infty \, \le \, C\Vert \phi \Vert_\infty.\]
Moreover, Control \eqref{eq:Control_in_inf_norm_deriv1} in Proposition \ref{prop:control_in_inf_norm} ensures
that
\[ \quad \Vert D_iP_t \phi(x_0+ \cdot)_{|E_h(\R^N)} \Vert_\infty \, \le \, C\Vert\phi \Vert_\infty \bigl(1+
t^{-\frac{1+\alpha(h-1)}{\alpha}}\bigr).\]
It follows immediately that
\[\Vert P_t\phi(x_0+\cdot)_{|E_h(\R^N)}\Vert_{C^1_b} \, \le \, C\Vert \phi \Vert_\infty(1+t^{-\frac{1+\alpha(h-1)}{\alpha}}).\]
We can now apply the interpolation inequality \eqref{eq:Interpolation_ineq} with $\delta_1=0$, $\delta_2=1$ and $r=\gamma/(1+\alpha(h-1))$
in order to obtain that
\[\Vert P_t\phi(x_0+\cdot)_{|E_h(\R^N)}\Vert_{C^r_b} \, \le \, C\Vert P_t\phi(x_0+\cdot)_{|E_h(\R^N)}\Vert^{r}_{C^1_b} \Vert
P_t\phi(x_0+\cdot)_{|E_h(\R^N)}\Vert^{1-r}_{\infty} \, \le \, C\Vert \phi\Vert_\infty\bigl(1+t^{-\frac{\gamma}{\alpha}}\bigr).\]
The argument is analogous for $\gamma$ in $(1,3) $, considering only the case $h=0$.
\end{proof}

The next result allows us to extend the controls in Proposition \ref{prop:control_in_inf_norm} to functions in the anisotropic
Zygmund-H\"older spaces. Roughly speaking, it states that the anisotropic $\gamma$-H\"older regularity induces a
"homogeneous" gain in time of order $\gamma/\alpha$ that can be used to weaken, at least partially, the time singularities associated
with the derivatives.
%These estimates for the derivatives of the semigroup $P_t$ will be fundamental to show the Schauder estimates in the following section.
The general argument of proof will mimic the one of Proposition \ref{prop:control_in_inf_norm} even if, this time, we will need to make the
H\"older modulus of $\phi$ appear. It will be managed introducing an auxiliary function $K$ (see Equation \eqref{eq:def_function_K}).

\begin{theorem}\label{prop:Control_in_Holder_norm}
Let $h,h',{h''}$ be in $\llbracket 1,n\rrbracket$ and $\phi$ in $C^\gamma_{b,d}(\R^N)$ for some $\gamma$ in $[0,1+\alpha)$. Then, there
exists a constant $C>0$ such that for any $i$ in $I_h$, any $j$ in $I_{h'}$ and any $k$ in $I_{h''}$, it holds that
\begin{align}
\label{eq:Control_in_Holder_norm_deriv1}
 \Vert D_i P_t \phi \Vert_\infty \, &\le \, C\Vert \phi \Vert_{C^\gamma_{b,d}}\bigl(1+t^{\frac{\gamma-(1+\alpha(h-1))}{\alpha}}\bigr),
  \quad t>0;\\
\label{eq:Control_in_Holder_norm_deriv2}
 \Vert D^2_{i,j}P_t \phi \Vert_\infty \, &\le \,C\Vert \phi\Vert_{C^\gamma_{b,d}}\bigl(1+t^{\frac{\gamma-(2+\alpha(h+h'-2))}{\alpha}}\bigr),
  \quad t>0;\\
\label{eq:Control_in_Holder_norm_deriv3}
 \Vert D^3_{i,j,k}P_t\phi\Vert_\infty \,&\le\,C\Vert\phi\Vert_{C^\gamma_{b,d}}\bigl(1+t^{\frac{\gamma-(3+\alpha(h+h'+h''-3))}{\alpha}}\bigr),
  \quad t>0.
\end{align}
\end{theorem}
\begin{proof}
Similarly to Proposition \ref{prop:control_in_inf_norm}, we start fixing a time horizon
\begin{equation}\label{def:eq:time_horizon_T}
T\,:= \,1\wedge T_0(N+6)>0.
\end{equation}
Then, Corollary \ref{coroll:C0_Cgamma_control} implies the continuity of $P_t$ on
$C^\gamma_{b,d}(\R^N)$,  for any $t\ge T/2$. Indeed,
\begin{equation}\label{Proof:eq:COntinuity_Pt_t>1}
\Vert P_t\phi \Vert_{C^\gamma_{b,d}} \, \le \, C\Vert \phi \Vert_\infty\bigl(1+t^{-\frac{\gamma}{\alpha}}\bigr)\, \le \, C_T\Vert \phi
\Vert_{C^\gamma_{b,d}}.
\end{equation}
The same argument shown in Equation \eqref{eq:COntrol_for_t>1} can now be applied to prove Control \eqref{eq:Control_in_Holder_norm_deriv1} for $t>T$. Namely,
\[\Vert D_i P_t \phi \Vert_\infty \, = \, \Vert D_i P_{T/2}\bigl(P_{t-T/2} \phi\bigr) \Vert_\infty \, \le \, C_T\Vert P_{t-T/2} \phi
\Vert_\infty \, \le \, C\Vert \phi \Vert_\infty.\]
The same reasoning can be used for the higher derivatives, too.\newline
When $t\le T$, let us assume $\alpha>1$, so that $1+\alpha>2$. The case $\alpha\le1$ can be handled similarly taking into
account one less derivative. Moreover, we notice that we need to prove Controls
\eqref{eq:Control_in_Holder_norm_deriv1}-\eqref{eq:Control_in_Holder_norm_deriv3} only for $\gamma$ in $(2,1+\alpha)$ thanks to
interpolation techniques. Indeed, if we want, for example, to prove Estimates \eqref{eq:Control_in_Holder_norm_deriv1} for some $\gamma'$ in
$(0,2]$, we can use Theorem \ref{Thm:Interpolation} to show that
\[\Vert D_iP_t \Vert_{\mathcal{L}_c(C^{\gamma'}_{b,d};B_b)} \, \le \, \bigl(\Vert D_iP_t
\Vert_{\mathcal{L}_c(B_b)}\bigr)^{1-\gamma'/\gamma}\bigl(\Vert D_iP_t \Vert_{\mathcal{L}_c(C^{\gamma}_{b,d},B_b)}\bigr)^{\gamma'/\gamma}\,
\le \, C\bigl(1+t^{\frac{\gamma'-(1+\alpha(h-1))}{\alpha}}\bigr),\]
once we have proven Estimate \eqref{eq:Control_in_Holder_norm_deriv1} for $\gamma>2$.\newline
We are only going to show Control \eqref{eq:Control_in_Holder_norm_deriv1} for $t \le T$ and $\gamma$ in $(2,1+\alpha)$. The estimates
\eqref{eq:Control_in_Holder_norm_deriv2}, \eqref{eq:Control_in_Holder_norm_deriv3} for higher derivatives can be obtained
analogously.\newline
Fixed $i$ in $I_h$ for some $h$ in $\llbracket 1,n\rrbracket$, we start writing from Equation \eqref{eq:Representation_semigroup} that
\[
D_iP_t\phi(x) \, = \, D_i\int_{\R^N}\int_{\R^N}\phi(\mathbb{M}_t(y_1+y_2)+e^{tA}x)p^{\text{tr}}(t,y_1)\,
dy_1\pi_t(dy_2).\]
Moreover, we introduce the function $K\colon \R^N\times \R^N\times \R^N\to \R$ given by
\begin{equation}\label{eq:def_function_K}
K(y_1,y_2,\xi) \, :=\, \phi(\mathbb{M}_ty_2+\xi)+\langle D_{I_1}\phi(\mathbb{M}_ty_2+\xi),E_1\mathbb{M}_ty_1\rangle+\frac{1}{2}\langle
D^2_{I_1}\phi(\mathbb{M}_ty_2+\xi)E_1\mathbb{M}_ty_1,E_1\mathbb{M}_ty_1\rangle.
\end{equation}
We then notice that the expression
\[\int_{\R^N}\int_{\R^N}p^{\text{tr}}(t,y_1)K(y_1,y_2,\xi)\, dy_1\pi_t(dy_2)\]
does not depend on $x$. Recalling $D_i$ stands for the derivative with respect to the variable $x_i$, we thus get that
\[D_i\int_{\R^N}\int_{\R^N}p^{\text{tr}}(t,y_1)K(y_1,y_2,\xi)\, dy_1\pi_t(dy_2) \, = \, 0, \quad \xi \in \R^N.\]
This property will allow to use a cancellation argument in Equation \eqref{Proof:eq:Control_diff0} below, once we split the small jumps in the non-degenerate contributions and the other ones. It now follows that
\[ D_iP_t\phi(x) \, = \, D_i\int_{\R^N}\int_{\R^N}\bigl[\phi(\mathbb{M}_t(y_1+y_2)+e^{tA}x)-K(y_1,y_2,\xi)\bigr]p^{\text{tr}}(t,y_1)\,
dy_1\pi_t(dy_2).\]
The same reasoning used in Equation \eqref{eq:Control_deriv1} can be applied here to show that, for any (fixed) $\xi$ in $\R^N$, it holds that
\begin{equation}\label{eq:Control_function_K}
\bigl{\vert}D_iP_t\phi(x)\bigr{\vert}  \, \le \, Ct^{-(h-1)} \int_{\R^N}\int_{\R^N} \bigl{\vert}\phi(\mathbb{M}_t(y_1+y_2)+e^{tA}x)-K(y_1,y_2,\xi)\bigr{\vert}\,\vert \nabla p^{\text{tr}} (t,
y_1)\vert\, dy_1\pi_t(dy_2).
\end{equation}
We fix now $\xi=e^{tA}x$ and denoting for simplicity $z=e^{tA}x+\mathbb{M}_ty_2$, we decompose the difference
between absolute values as
\begin{equation}\label{Proof:eq:Control_diff0}
\begin{split}
\bigl{\vert} \phi(\mathbb{M}_ty_1+z)- K(y_1,y_2,e^{tA}x)\bigr{\vert} \, \le \,
\bigl{\vert} \phi(\mathbb{M}_ty_1+z)- \phi(E_1\mathbb{M}_ty_1+z)\bigr{\vert}+\bigl{\vert} &\phi(E_1\mathbb{M}_ty_1+z)
-K(y_1,y_2,e^{tA}x)\bigr{\vert}\\
&=:\, \Lambda_1+\Lambda_2.
\end{split}
\end{equation}
While it is trivial to control the first term as
\begin{equation}\label{Proof:eq:Control_diff1}
\Lambda_1 \, \le \, \Vert \phi \Vert_{C^\gamma_{b,d}}\sum_{k=2}^{n}\vert E_k\mathbb{M}_ty_1\vert^{\frac{\gamma}{1+\alpha(k-1)}},
\end{equation}
the second component $\Lambda_2$ needs a Taylor expansion to write that
\begin{equation}\label{Proof:eq:Control_diff2}
\begin{split}
\Lambda_2 \, &= \, \bigl{\vert}\phi(E_1\mathbb{M}_ty_1+z)- \phi(z)-\langle D_{I_1}\phi(z),E_1\mathbb{M}_ty_1\rangle-\frac{1}{2}\langle
D^2_{I_1}\phi(z)E_1\mathbb{M}_ty_1,E_1\mathbb{M}_ty_1\rangle\bigr{\vert}\\
&= \,\bigl{\vert} \int_{0}^{1}\langle D^2_{I_1}\phi(z+ \lambda E_1\mathbb{M}_ty_1)E_1\mathbb{M}_ty_1,
E_1\mathbb{M}_ty_1\rangle(1-\lambda) \, d\lambda - \frac{1}{2}\langle D^2_{I_1}\phi(z) E_1\mathbb{M}_ty_1, E_1\mathbb{M}_ty_1\rangle
\bigr{\vert} \\
&\le \, \Vert D^2_{I_1}\phi \Vert_{C^{\gamma-2}_b(E_1)}  \vert E_1\mathbb{M}_ty_1\vert^{\gamma-2} \vert E_1\mathbb{M}_ty_1\vert^2 \\
&\le \, \Vert \phi \Vert_{C^\gamma_{b,d}}\vert E_1\mathbb{M}_ty_1\vert^\gamma.
\end{split}
\end{equation}
Going back to Expression \eqref{eq:Control_function_K} with Estimates \eqref{Proof:eq:Control_diff0}, \eqref{Proof:eq:Control_diff1} and
\eqref{Proof:eq:Control_diff2}, we can show that
\[\bigl{\vert}D_iP_t\phi(x)\bigr{\vert} \, \le \, C\Vert \phi \Vert_{C^\gamma_{b,d}}t^{-(h-1)}\sum_{k=1}^{n}\int_{\R^N}  \vert
E_k\mathbb{M}_ty_1\vert^{\frac{\gamma}{1+\alpha(k-1)}} \vert
\nabla p^{\text{tr}} (t,y_1)\vert\, dy_1.\]
The above expression allows us to conclude as in \eqref{eq:Control_deriv3} using Proposition \ref{prop:control_on_truncated_density} with
$m=N+4$ and $\vert\vartheta\vert=1$. Namely,
\[\begin{split}
\bigl{\vert}D_iP_t\phi(x)\bigr{\vert}\, &\le \,  C\Vert \phi \Vert_{C^\gamma_{b,d}}t^{-(h-1)}\sum_{k=1}^{n}\int_{\R^N}
t^{-\frac{N+1}{\alpha}}\bigl(1+ \frac{\vert y_1\vert}{t^{\frac{1}{\alpha}}} \bigr)^{-(N+3)} \vert
E_k\mathbb{M}_ty_1\vert^{\frac{\gamma}{1+\alpha(k-1)}} \, dy_1 \\
&\le \, C \Vert \phi \Vert_{C^\gamma_{b,d}} t^{\frac{\gamma-(1+\alpha(h-1))}{\alpha}}\sum_{k=1}^{n}\int_{\R^N}\bigl(1+\vert z
\vert\bigr)^{-(N+3)} \vert z\vert^{\frac{\gamma}{1+\alpha(k-1)}} \, dz \\
&\le \,  C \Vert \phi \Vert_{C^\gamma_{b,d}}t^{\frac{\gamma -(1+\alpha(h-1))}{\alpha}},
\end{split}\]
where in the second step we used again the change of variable $z=y_1t^{-1/\alpha}$.
\end{proof}

Next, we are going to use the controls in Theorem \ref{prop:Control_in_Holder_norm} to show the main result of this section. It states
the continuity of the semigroup $P_t$ between anisotropic Zygmund-H\"older spaces at a cost of additional time singularities.

\begin{corollary} \label{corollary:continuity_between_holder}
Let $\beta,\gamma$ be in $[0,1+\alpha)$ such that $\beta\le \gamma$. Then, there exists a constant $C>0$ such that
\begin{equation}\label{eq:coroll:Cbeta_Cgamma_control}
\Vert P_t \Vert_{\mathcal{L}_c(C^\beta_{b,d},C^\gamma_{b,d})} \, \le \, C\bigl(1+t^{\frac{\beta-\gamma}{\alpha}}\bigr), \quad
t>0.
\end{equation}
\end{corollary}
\begin{proof}
It is enough to show the result only for $\gamma=\beta$ non-integer, thanks to interpolation techniques. Indeed,
fixed $\beta< \gamma$, we can use Theorem \ref{Thm:Interpolation} to show that
\[\Vert P_t\Vert_{\mathcal{L}_c(C^\beta_{b,d}(\R^N),C^{\gamma}_{b,d}(\R^N))} \,\le \, \Bigl(\Vert
P_t\Vert_{\mathcal{L}_c(C_b(\R^N),C^{\gamma}_{b,d}(\R^N))}\Bigr)^{1-\frac{\beta}{{\gamma}}}\Bigl(\Vert
P_t\Vert_{\mathcal{L}_c(C^{\gamma}_{b,d}(\R^N))}\Bigr)^{\frac{\beta}{{\gamma}}}.\]
On the other hand, if we fix $\gamma$ integer, we can take $\gamma'$ in $(\gamma,1+\alpha)$ non-integer such that Theorem \ref{Thm:Interpolation} implies:
\[\Vert P_t\Vert_{\mathcal{L}_c(C^{\gamma}_{b,d}(\R^N))} \, \le  \Bigl(\Vert
P_t\Vert_{\mathcal{L}_c(C_b(\R^N)}\Bigr)^{1-\frac{\gamma}{\gamma'}}\Bigl(\Vert
P_t\Vert_{\mathcal{L}_c(C^{\gamma'}_{b,d}(\R^N))}\Bigr)^{\frac{\gamma}{\gamma'}}.\]
The general result will then follows from the two above controls and Equation \eqref{eq:coroll:C0_Cgamma_control}, once we
have shown  Estimate \eqref{eq:coroll:Cbeta_Cgamma_control}  for $\gamma=\beta$ non-integer.\newline
Fixed again the time horizon $T$ given in \eqref{def:eq:time_horizon_T}, we start noticing that Control \eqref{eq:coroll:Cbeta_Cgamma_control} for $t\ge T$ has already been shown in Equation \eqref{Proof:eq:COntinuity_Pt_t>1} .\newline
To prove it when $t\le T$, we are going to exploit the equivalent norm defined in \eqref{eq:def_equivalent_norm} of Lemma \ref{lemma:def_equivalent_norm}. For this reason, we fix $h$ in $\llbracket 1,n \rrbracket$, a point $x_0$ in $\R^N$ and
$z\neq 0$ in $E_h(\R^N)$ and we would like to show that
\begin{equation}\label{proof_eq:Corollary1}
\bigl{\vert} \Delta^3_{x_0}\bigl(P_t\phi\bigr)(z) \bigr{\vert}  \, \le \, C\Vert \phi\Vert_{C^\gamma_{b,d}}\vert
z\vert^{\frac{\gamma}{1+\alpha(h-1)}},
\end{equation}
for some constant $C>0$ independent from $x_0$. Before starting with the calculations, we highlight the presence of three
different "regimes" appearing below. On the one hand, we will firstly consider a \emph{macroscopic regime} appearing for $\vert z \vert\ge
1$. On the other hand, we will say that the \emph{off-diagonal regime} holds if
$t^{\frac{1+\alpha(h-1)}{\alpha}} \le \vert z \vert \le 1$. It will mean in particular that the spatial distance is larger than the
characteristic time-scale. Finally, a \emph{diagonal regime} will be in force when $t^{\frac{1+\alpha(h-1)}{\alpha}} \ge \vert z \vert$ and
the spatial point will be instead smaller than the typical time-scale magnitude. While for the two first regimes, we are going to use the
contraction property of the semigroup, the third regime will require to exploit the controls in H\"older norms given by Theorem
\ref{prop:Control_in_Holder_norm}. \newline
As said above, Estimate \eqref{proof_eq:Corollary1} in the macroscopic regime (i.e.\ $\vert z \vert \ge 1$) follows immediately from
the contraction property of $P_t$ on $B_b(\R^N)$. Indeed,
\begin{equation}\label{proof_eq:Corollary1-1}
\bigl{\vert}\Delta^3_{x_0}\bigl(P_t\phi\bigr)(z) \bigr{\vert} \, \le \, C\Vert P_t\phi \Vert_\infty \, \le \, C\Vert
\phi\Vert_{C^\gamma_{b,d}}\vert z\vert^{\frac{\gamma}{1+\alpha(h-1)}}.
\end{equation}
For $t^{\frac{1+\alpha(h-1)}{\alpha}} \le \vert z \vert \le 1$ and $l$ in $\llbracket 0,3\rrbracket$, we start noticing from Equation
\eqref{eq:Representation_semigroup} that
\[
\begin{split}
P_t\phi(x_0+lz) \, &=\, \int_{\R^N}\int_{\R^N}\phi\bigl(\mathbb{M}_t(y_1+y_2)+e^{tA}(x_0+lz)\bigr) p^{\text{tr}}(t,y_1)\, dy_1\pi_t(dy_2)\\
&=\, \int_{\R^N}\int_{\R^N}\phi\bigl(\xi_0+le^{tA}z)\bigr) p^{\text{tr}}(t,y_1)\, dy_1\pi_t(dy_2),
\end{split}
\]
where we have denoted for simplicity $\xi_0=\mathbb{M}_t(y_1+y_2)+e^{tA}x_0$. We can then exploit Lemma \ref{lemma:Decomposition_exp_A} to
write that
\begin{equation}\label{proof_eq:Corollary1-2}
\begin{split}
\bigl{\vert}\Delta^3_{x_0}\bigl(P_t\phi\bigr)(z) \bigr{\vert} \, &\le \,
\int_{\R^N}\int_{\R^N}\bigl{\vert}\Delta^3_{\xi_0}\phi(e^{tA}z) \bigr)\bigr{\vert}
p^{\text{tr}}(t,y_1)\,dy_1\pi_t(dy_2) \\
&\le \,  \pi_t(\R^N)\Vert \phi \Vert_{C^\gamma_{b,d}}\sum_{k=1}^{n}\vert E_ke^{tA}z\vert^{\frac{\gamma}{1+\alpha(k-1)}} \\
&\le \, C\Vert \phi \Vert_{C^\gamma_{b,d}}\Bigl[\sum_{k=1}^{h-1}\bigl(t\vert
z\vert\bigr)^{\frac{\gamma}{1+\alpha(k-1)}}+\sum_{k=h}^{n}\bigl(t^{k-h}\vert z\vert\bigr)^{\frac{\gamma}{1+\alpha(k-1)}}\Bigr] \\
&\le C\Vert \phi \Vert_{C^\gamma_{b,d}}\vert z \vert^{\frac{\gamma}{1+\alpha(h-1)}}.
\end{split}
\end{equation}
For $\vert z \vert\le t^{\frac{1+\alpha(h-1)}{\alpha}}$, we are going to apply Taylor expansion three times in order to make $D^3_{I_h}$
appear. Namely,
\begin{equation}\label{Proof:eq:Taylor_Expansion}
\begin{split}
\bigl{\vert}\Delta^3_{x_0}\bigl(P_t\phi\bigr)(z)\bigr{\vert}\, &= \, \Bigl{\vert}\int_{0}^{1} \langle
D_{I_h}P_t\phi(x_0+\lambda z)-2D_{I_h}P_t\phi(x_0+z+\lambda z)+D_{I_h}P_t\phi(x_0+2z+\lambda z), z \rangle \,d\lambda\Bigr{\vert} \\
&\le \, \Bigl{\vert}\int_{0}^{1}\int_{0}^{1} \langle \bigl[D^2_{I_h}P_t\phi(x_0+(\lambda+\mu)z)-D^2_{I_h}P_t\phi(x_0+z+(\lambda+\mu)z)\bigr]z,
z \rangle \,d\lambda d\mu\Bigr{\vert} \\
&\le \, \Bigl{\vert}\int_{0}^{1}\int_{0}^{1}\int_{0}^{1} \langle \bigl[D^3_{I_h}P_t\phi(x_0+(\lambda+\mu+\nu)z)\bigr](z,z),z \rangle
\,d\lambda d\mu d\nu\Bigr{\vert} \\
&\le \, C\Vert D^3_{I_h}P_t\phi\Vert_\infty \vert z \vert^3 \\
&\le \, C\Vert \phi \Vert_{C^{\gamma}_{b,d}}\bigl(1+t^{\frac{\gamma-3(1+\alpha(h-1))}{\alpha}}\bigr) \vert z \vert^3,
\end{split}
\end{equation}
where in the last step we used Control \eqref{eq:Control_in_Holder_norm_deriv3} with $h=h'={h''}$. Since $\vert z \vert \le
t^{\frac{1+\alpha(h-1)}{\alpha}}$ and noticing that $\gamma-3(1+\alpha(h-1))<0$, it holds that
\[\bigl(1+t^{\frac{\gamma-3(1+\alpha(h-1))}{\alpha}}\bigr)\vert z \vert^3 \, \le \, \vert z
\vert^{\frac{\gamma-3(1+\alpha(h-1))}{1+\alpha(h-1)}}\vert z \vert^3 \, = \, \vert z \vert^{\frac{\gamma}{1+\alpha(h-1)}}.\]
We can then conclude that
\begin{equation}\label{proof_eq:Corollary1-3}
\bigl{\vert}\Delta^3_{x_0}\bigl(P_t\phi\bigr)(z)\bigr{\vert}\, \le \, C\Vert \phi \Vert_{C^\gamma_{b,d}}\vert
z\vert^{\frac{\gamma}{1+\alpha(h-1)}}.
\end{equation}
Going back to Controls \eqref{proof_eq:Corollary1-1}, \eqref{proof_eq:Corollary1-2} and \eqref{proof_eq:Corollary1-3}, we have thus proven
Estimate \eqref{proof_eq:Corollary1} for any non-integer $\gamma=\beta$.
\end{proof}

\setcounter{equation}{0}
\section{Elliptic and Parabolic Schauder Estimates}
In this section, we use the controls shown before to prove Schauder Estimates both for the elliptic and the parabolic equation driven by the Ornstein-Ulhenbeck operator $\mathcal{L}^{\text{ou}}$.

Fixed $\lambda>0$ and $g$ in $C_b(\R^N)$, we say that a function $u\colon \R^N\to \R^N$ is a \emph{weak solution} of Elliptic
Equation \eqref{eq:Elliptic_IPDE} if $u$ is in $C_b(\R^N)$ and for any $\phi$ in $C^\infty_c(\R^N)$ (i.e.\ smooth functions with compact
support), it holds that
\begin{equation}\label{eq:weak_solution_elliptic}
\int_{\R^N}u(x)\bigl[\lambda\phi(x)-\bigl(\mathcal{L}^{ou}\bigr)^*\phi(x)\bigr] \, dx \, = \,
\int_{\R^N}\phi(x)g(x) \, dx,
\end{equation}
where $\bigl(\mathcal{L}^{ou}\bigr)^*$ denotes the formal adjoint of $\mathcal{L}^{ou}$ on $L^2(\R^N)$, i.e.
\begin{equation}\label{eq:def_adjoint_op}
\bigl(\mathcal{L}^{ou}\bigr)^*\phi(x) \, = \, \mathcal{L}^\ast \phi(x) -\langle Ax,D_x \phi(x)\rangle - \text{Tr}(A)\phi(x), \quad
(t,x) \in [0,T]\times \R^N,
\end{equation}
and $\mathcal{L}^\ast$ is the adjoint of the operator $\mathcal{L}$ on $L^2(\R^N)$. It is well-known (see e.g. Section $4.2$ in
\cite{book:Applebaum19}) that it can be represented for any $\phi$ in $C^\infty_c(\R^N)$ as
\[\mathcal{L}^\ast\phi(x) \, = \,\frac{1}{2}\text{Tr}\bigl(BQB^\ast D^2\phi(x)\bigr) -\langle Bb,D \phi(x)\rangle +
\int_{\R^d_0}\bigl[\phi(x-Bz)-\phi(x)+\langle D\phi(x),
Bz\rangle\mathds{1}_{B(0,1)}(z) \bigr] \,\nu(dz).\]
We state now the main result for the elliptic case, ensuring the well-posedness (in a weak sense) for Equation \eqref{eq:Elliptic_IPDE}.

\begin{theorem}
\label{thm:well_posedness_elliptic}
Fixed $\lambda>0$, let $g$ be in $C_b(\R^N)$. Then, the function $u\colon \R^N\to \R$ given by
\begin{equation}\label{eq:representation_elliptic_sol}
u(x) \, := \, \int_{0}^{\infty}e^{-\lambda t}P_tg(x) \, dt, \quad x \in \R^N,
\end{equation}
is the unique weak solution of Equation \eqref{eq:Elliptic_IPDE}.
\end{theorem}
\begin{proof}
\emph{Existence.} We are going to show that the function $u$ given in Equation \eqref{eq:representation_elliptic_sol} is indeed a
weak solution of the elliptic problem \eqref{eq:Elliptic_IPDE}. It is straightforward to notice that $u$ is in $C_b(\R^N)$, thanks to the contraction property of $P_t$ on $C_b(\R^N)$. Fixed $\phi$ in $C^\infty_c(\R^N)$, we then use Fubini Theorem to write that
\[
\begin{split}
\int_{\R^N}u(x)\bigl(\mathcal{L}^{\text{ou}}\bigr)^\ast\phi(x) \, dx \, &= \,
\lim_{\epsilon\to0^+}\int_{\epsilon}^{\infty}\int_{\R^N}e^{-\lambda t}
P_tg(x)\bigl(\mathcal{L}^{\text{ou}}\bigr)^\ast\phi(x) \, dx dt \\
&= \,\lim_{\epsilon\to0^+}\int_{\epsilon}^{\infty}\int_{\R^N}e^{-\lambda t}
\mathcal{L}^{\text{ou}}P_tg(x)\phi(x) \, dx dt,
\end{split}
\]
where, in the last step, we exploited that $P_tg$ is differentiable and bounded for $t>0$ (Proposition
\ref{prop:control_in_inf_norm}). Since $\mathcal{L}^{\text{ou}}$ is the infinitesimal generator of the semigroup $\{P_t\colon
t\ge0\}$, we know that $\partial_t(P_tg)$ exists for any $t>0$ and $\partial_t(P_tg)(x)=\mathcal{L}^{\text{ou}}P_tg(x)$ for any $x$ in
$\R^N$. Integration by parts formula allows then to conclude that
\[
\begin{split}
\int_{\R^N}u(x)\bigl(\mathcal{L}^{\text{ou}}\bigr)^\ast\phi(x) \, dx \, &=
\,\lim_{\epsilon\to0^+}\int_{\R^N}\phi(x)\int_{\epsilon}^{\infty}e^{-\lambda
t}\partial_tP_tg(x) \, dt dx \\
&= \, \lim_{\epsilon\to0^+}\int_{\R^N}\Bigl(-e^{-\lambda \epsilon}P_\epsilon g(x)+\lambda\int_{\epsilon}^{\infty}e^{-\lambda t}P_t g(x) \,
dt\Bigr) \, dx \\
&= \, \int_{\R^N}-g(x)\phi(x) \, dx + \int_{\R^N}\lambda u(x)\phi(x) \, dx.
\end{split}
\]

%%%%%%%%%%%%%%%%%%%%%%%%%%%%%%%%%%%%%%%%%%%%%%%%%%%%%
\emph{Uniqueness.} It is enough to show that any weak solution $u$ of Equation \eqref{eq:Elliptic_IPDE} for $g=0$ coincides with
the zero function, i.e.\ $u=0$. To do so, we fix a function $\rho$ in $C^{\infty}_c(\R^N)$ such that $\Vert
\rho \Vert_{L^1}=1$, $0\le \vert \rho \vert \le 1$ and we then define the \emph{mollifier} $\rho_m:=m^N\rho(mx)$ for any $m$ in $\N$. Denoting now, for simplicity, $u_m:=u\ast \rho_m$, we define the function
\begin{equation}\label{Proof:eq:def_gm}
g_m(x) \, := \, \lambda u_m(x)-\mathcal{L}^{\text{ou}}u_m(x).
\end{equation}
Using that $u$ is in $C_b(\R^N)$, it is easy to notice that $g_m$ is also in $C_b(\R^N)$ for any fixed $m$ in $\N$.
Truncating the functions if necessary, we can assume that $u_m$ and $g_m$ are integrable with integrable Fourier transform so that we can
apply the Fourier transform in Equation \eqref{Proof:eq:def_gm}:
\begin{equation}\label{proof:Schauder_elliptic1}
\lambda\widehat{u}_m(\xi)- \mathcal{F}_x \bigl(\mathcal{L}^{\text{ou}} u_m\bigr)(\xi) \, = \, \widehat{g}_m(\xi).
\end{equation}
We remember in particular that the above operator $\mathcal{L}^{\text{ou}}$ has an associated L\'evy symbol
$\Psi^{ou}(\xi)$ and, following Section $3.3.2$ in \cite{book:Applebaum09}, it holds that
\begin{equation}\label{proof:Schauder_elliptic2}
\mathcal{F}_x \bigl(\mathcal{L}^{\text{ou}} u_m\bigr)(\xi) \, = \,
\Psi^{ou}(\xi)\widehat{u}_m(\xi).
\end{equation}
We can then use it to show that $\widehat{u}_m$ is a classical solution of the following equation:
\[\bigl[\lambda- \Psi^{ou}(\xi)\bigr]\widehat{u}_m(\xi) \, = \, \widehat{g}_m(\xi).\]
The above equation can be easily solved by direct calculation as
\[\widehat{u}_m(\xi) \, = \, \int_{0}^{\infty} e^{-\lambda t}e^{t\Psi^{ou}(\xi)} \widehat{g}_m(\xi)\, ds.\]
In order to go back to $u_m$, we apply now the inverse Fourier transform to write that
\[u_m(x) \, = \, \int_{0}^{\infty}e^{-\lambda t}P_{t}g_m(x) \, dt.\]
We can then exploit the contraction property of the semigroup $P_t$  to show that $\Vert u_m \Vert_\infty \le C \Vert g_m \Vert_\infty$.
In order to conclude, we need to show that
\begin{equation}\label{Proof:convergence_gm}
\lim_{m \to \infty}\Vert g_m \Vert_{\infty} \, = \,  0.
\end{equation}
We start noticing that, since $u$ is a weak solution of Equation \eqref{eq:Elliptic_IPDE} with $g=0$, it holds that
\[
\begin{split}
g_m(x) \, &= \, \int_{\R^N}u(y)\bigl{\{}\lambda \rho_m(x-y) - \mathcal{L}[\rho_m(\cdot-y)](x)-\langle Ax,D_x\rho_m(x-y)\rangle\bigr{\}} \,
dy\\
&= \, \int_{\R^N}u(y)\bigl{\{}\mathcal{L}^\ast[\rho_m(x-\cdot)](y) - \mathcal{L}[\rho_m(\cdot-y)](x)+\langle
A(x-y),D_x\rho_m(x-y)\rangle+\text{Tr}(A)\rho_m(x-y)\bigr{\}} \, dy\\
&= \, R^1_m(x)+R^2_m(x)+R^3_m(x),
\end{split}\]
where we have denoted
\begin{align*}
  R^1_m(x) \, &:= \, \int_{\R^N}u(y)\bigl[\mathcal{L}^\ast[\rho_m(x-\cdot)](y) - \mathcal{L}[\rho_m(\cdot-y)](x)\bigr] \, dy;  \\
  R^2_m(x) \, &:= \,  \int_{\R^N}u(y)\langle A(x-y),D_x\rho_m(x-y)\rangle \, dy;\\
  R^3_m(x) \, &:= \,  \int_{\R^N}u(y)\text{Tr}(A)\rho_m(x-y) \, dy.
\end{align*}
On the one hand, it is easy to notice that $R^1_m=0$, since $\mathcal{L}^\ast[\rho_m(x-\cdot)](y)=\mathcal{L}[\rho_m(\cdot-y)](x)$ for any
$m$ in $\N$ and any $y$ in $\R^N$. Indeed, it holds that
\[\frac{1}{2}\text{Tr}\bigl(BQB^\ast D^2_y[\rho_m(x-\cdot)](y)\bigr)-\langle Bb,D_y[\rho_m(x-\cdot)](y)\rangle \, = \,
\frac{1}{2}\text{Tr}\bigl(BQB^\ast D^2_x\rho_m(x-y)\bigr)+\langle Bb,D_x\rho_m(x-y)\rangle\]
and
\begin{multline*}
  \int_{\R^d_0}\bigl[\rho_m(x-y+Bz)-\rho_m(x-y)+\langle
D_y[\rho_m(x-\cdot)](y),Bz\rangle\mathds{1}_{B(0,1)}(z)\bigr]\, \nu(dz) \\
= \, \int_{\R^d_0}\bigl[\rho_m((x+Bz)-y)-\rho_m(x-y)-\langle
D_x\rho_m(x-y),Bz\rangle\mathds{1}_{B(0,1)}(z)\bigr]\, \nu(dz).
\end{multline*}
On the other hand, it can be checked (see e.g. \cite{Priola09}) that $\Vert R^2_m+R^3_m\Vert_\infty\to 0$ if $m$ goes to infinity.
Indeed, we firstly notice that $R^3_m$ converges, when $m$ goes to infinity, to the function
$u\text{Tr}(A)$, uniformly in $x$. On the other hand, applying the change of variables $y=x-z/m$ in $R^2_m$, we can obtain that
\[R^2_m(x) \, = \,  m\int_{\R^N}u(x-z/m)\langle A(z/m),D_x\rho(z)\rangle \, dy.\]
Letting $m$ goes to infinity above, we can then conclude that $R^2_m$ converges to the function $-u\text{Tr}(A)$, uniformly in $x$.
\end{proof}

Let us deal now with the parabolic setting. Since we are working with evolution equations, the functions we consider will often depend
on time, too. We denote for any $\gamma>0$ the space $L^\infty(0,T,C^{\gamma}_{b,d}(\R^{N}))$ as the family of functions
$\phi$ in $B_b\bigl([0,T]\times \R^{N})$ such that $\phi(t,\cdot)$ is in $C^{\gamma}_{b,d}(\R^N)$ at any fixed $t$ and the norm
\[\Vert \phi \Vert_{L^\infty(C^\gamma_{b,d})} \, := \, \sup_{t\in [0,T]}\Vert \phi(t,\cdot)\Vert_{C^\gamma_{b,d}} \text{ is finite.}\]

We define now the notion of solution we are going to consider. Fixed $T>0$, $u_0$ in $C_b(\R^N)$ and $f$ in
$L^\infty\bigl(0,T;C_b(\R^N)\bigr)$, we
say that a function $u\colon [0,T]\times\R^N\to \R^N$ is a \emph{weak solution} of the Cauchy problem \eqref{eq:Parabolic_IPDE} if $u$ is in
$L^\infty\bigl(0,T;C_b(\R^N)\bigr)$ and for any $\phi$ in $C^\infty_c([0,T)\times\R^N)$, it holds that
\begin{equation}\label{eq:weak_solution_Parabolic}
\int_{\R^N}u_0(x)\phi(0,x)\, dx + \int_{0}^{T}\int_{\R^N}u(t,x)\Bigl[\partial_t\phi(t,x)+\bigl(\mathcal{L}^{ou}\bigr)^*\phi(t,x)\Bigr]
+f(t,x)\phi(t,x)\, dxdt \, = \, 0,
\end{equation}
where $\bigl(\mathcal{L}^{ou}\bigr)^*$ denotes the formal adjoint of $\mathcal{L}^{ou}$ on $L^2(\R^N)$ given in Equation
\eqref{eq:def_adjoint_op}.

Similarly to the elliptic setting, we show firstly the weak well-posedness of the Cauchy problem \eqref{eq:Parabolic_IPDE}.

\begin{theorem}
\label{thm:well_posedness_parabolic}
Fixed $T>0$, let $u_0$ be a function in $C_b(\R^N)$ and $f$ in $L^\infty\bigl(0,T;C_b(\R^N)\bigr)$. Then, the function ${u\colon [0,T]\times\R^N\to
\R}$ given by
\begin{equation}\label{eq:representation_parabolic_sol}
u(t,x) \, := \, P_tu_0(x) + \int_{0}^{t}P_{t-s}f(s,x) \, ds, \quad (t,x) \in [0,T]\times\R^N,
\end{equation}
is the unique weak solution of the Cauchy problem \eqref{eq:Parabolic_IPDE}.
\end{theorem}
\begin{proof}
\emph{Existence.} We start considering a "regularized" version of the coefficients appearing in Equation \eqref{eq:Parabolic_IPDE}.
Namely, we consider a family $\{u_{0,m}\}_{m\in \N}$ in $C^\infty_b(\R^N)$ such that $u_{0,m}\to u_0$ uniformly in
$x$ and a family $\{f_m\}_{m\in \N}$ in $L^\infty\bigl(0,T;C^\infty_b(\R^N)\bigr)$ such that $f_m\to f$ uniformly in $t$ and $x$. They can
be obtained through standard mollification methods in space.\newline
Fixed $m$ in $\N$, we denote now by $u_m\colon[0,T]\times\R^n\to\R$ the function given by
\[u_m(t,x) \, := \, P_tu_{0,m}(x) + \int_{0}^{t}P_{t-s}f_m(s,x) \, ds, \quad t \in [0,T], x \in \R^N.\]
On the one hand, we use again that $\partial_t(P_tu_m)(t,x)=\mathcal{L}^{\text{ou}}P_tu_m(t,x)$ for any $(t,x)$ in
$[0,T]\times\R^N$ to check that $u_m$ is indeed a \emph{classical} solution of the "regularized" Cauchy Problem:
\[
\begin{cases}
  \partial_tu_m(t,x) \, = \, \mathcal{L}^{\text{ou}}u_m(t,x)+f_m(t,x), \quad (t,x) \in (0,T)\times \R^N; \\
  u_m(0,x) \, = \, u_{0,m}(x), \quad x \in \R^N.
\end{cases}
\]
On the other hand, we exploit the linearity and the continuity of the semigroup $P_t$ on $C_b(\R^N)$ to show that
\[u_m\, = \, P_tu_{0,m}(x) + \int_{0}^{t}P_{t-s}f_m(s,x) \, ds \,  \overset{m}{\to} \, P_tu_0(x) + \int_{0}^{t}P_{t-s}f(s,x) \, ds \, = \, u,\]
uniformly in $t$ and $x$, where $u$ is the function given in \eqref{eq:representation_parabolic_sol}. \newline
We fix now a test function $\phi$ in $C^\infty_0\bigl([0,T)\times\R^N\bigr)$ and we then notice that
\[\int_{0}^{T}\int_{\R^N}\phi(t,y)\Bigl(\partial_t-\mathcal{L}^{ou}\Bigr)u_m(t,y) \, dydt \, = \,
\int_{0}^{T}\int_{\R^N}\phi(t,y)f_m(t,y) \, dydt.\]
An integration by parts allows now to move the operator to the test function, being careful to remember that $u_m(0,\cdot)=u_{0,m}(\cdot)$.
Indeed, it holds that
\begin{equation}\label{TOweak:1}
-\int_{0}^{T}\int_{\R^N}\Bigl(\partial_t+\bigl(\mathcal{L}^{ou}\bigr)^*\Bigr)\phi(t,y)u_m(t,y) \, dydt \, = \,
\int_{\R^N}\phi(0,y)u_{0,m}(y) \, dy + \int_{0}^{T}\int_{\R^N}\phi(t,y)f_m(t,y) \, dydt,
\end{equation}
where $\bigl(\mathcal{L}^{ou}\bigr)^*$ denotes the formal adjoint of $\mathcal{L}^{ou}$ on $L^2(\R^N)$.\newline
We would like now to go back to the solution $u$, letting $m$ go to infinity. We start rewriting the right-hand side term of \eqref{TOweak:1} as $R^1_m+R^2_m$,
where
\begin{align*}
R^1_m \, &:= \,\int_{\R^N}\phi(0,y)u_{0,m}(y) \, dy;   \\
R^2_m \, &:= \, \int_{0}^{T}\int_{\R^N}\phi(t,y)f_m(t,y) \, dydt.
\end{align*}
We can rewrite $R^2_m$ as
\[R^2_m \, = \, \int_{0}^{T}\int_{\R^N}\phi(t,y)f(t,y) \, dydt + \int_{0}^{T}\int_{\R^N}\phi(t,y)\bigl[f_m-f\bigr](t,y) \, dydt.\]
Exploiting that, by assumption, $f_m$ converges to $f$ uniformly in $t$ and $x$, it is easy to see that the second
contribution above converges to $0$. A similar argument can be used to show that
\[\int_{\R^N}\phi(0,y)u_{0,m}(y) \, dy \, \overset{m}{\to} \, \int_{\R^N}\phi(0,y)u_{0}(y) \, dy.\]
On the other hand, we can rewrite the left-hand side of Equation \eqref{TOweak:1} as
\[-\int_{0}^{T}\int_{\R^N}\Bigl(\partial_t+\bigl(\mathcal{L}^{ou}\bigr)^*\Bigr)\phi(t,y)u_m(t,y) \, dydt \,  = \,
-\int_{0}^{T}\int_{\R^N}\Bigl(\partial_t+\bigl(\mathcal{L}^{ou}\bigr)^*\Bigr)\phi(t,y)u(t,y) \, dydt +L^1_m+L^2_m+L^3_m,\]
where we have denoted
\begin{align}
L^1_m \, &:= \, \int_{0}^{T}\int_{\R^N} \bigl[\frac{1}{2}\text{Tr}\bigl(BQB^\ast D^2_y\phi(t,y)\bigr)+\langle Ay+Bb,D_y\phi(t,y)\rangle+\text{Tr}(A)\phi(t,y)\bigr][u_m-u](t,y) \, dydt;   \notag\\
L^2_m \, &:= \, \int_{0}^{T}\int_{\R^N}\partial_t\phi(t,y)[u-u_m](t,y) \, dydt; \label{eq:remainder_in_existence}\\
L^3_m \, &:= \, \int_{0}^{T}\int_{\R^N}[u-u_m](t,y)\Bigl[\int_{\R^d_0}\phi(t,y-Bz)-\phi(t,y)+\langle
D_y\phi(t,y),Bz\rangle\mathds{1}_{B(0,1)}(z)\, \nu(dz)\Bigr]dydt. \notag
\end{align}
To conclude, we need to show that the remainder $L^1_m+L^2_m+L^3_m$ is negligible, if $m$ goes to infinity. Exploiting that $\phi$ has a
compact support and that $\Vert u_m - u\Vert_\infty\overset{m}{\to} 0$, it is easy to show that
$\vert L^1_m+L^2_m\vert \overset{m}{\to} 0$.\newline
In order to control $L^3_m$, we need firstly to decompose it as $L^{3,1}_m+L^{3,2}_m$, where
\begin{align*}
L^{3,1}_m \, &:= \, \int_{0}^{T}\int_{\R^N}\bigl[u-u_m\bigr](t,y)\Bigl[\int_{0<\vert z\vert<1}\phi(t,y-Bz)-\phi(t,y)+\langle
D_y\phi(t,y),Bz\rangle\, \nu(dz)\Bigr]dydt;   \\
L^{3,2}_m \, &:= \, \int_{0}^{T}\int_{\R^N}\bigl[u-u_m\bigr](t,y)\Bigl[\int_{\vert z\vert>1}\phi(t,y-Bz)-\phi(t,y)\, \nu(dz)\Bigr]dydt.
\end{align*}
The second term $L^{3,2}_m$ can be controlled easily using the Fubini Theorem. Indeed, denoting by $K$ the support of $\phi$ and by $\lambda$ the Lebesgue measure on $\R^N$, we notice that
\[
\begin{split}
\vert L^{3,2}_m\vert \, &\le \, \Vert u-u_m\Vert_\infty\int_{0}^{T}\int_{\vert z\vert>1}\int_{\R^N}\vert\phi(t,y-Bz)-\phi(t,y)\vert\,
dy\nu(dz)dt\\
&\le \,  CT2\lambda(K)\nu\bigl(B^c(0,1)\bigr)\Vert u-u_m\Vert_\infty.
\end{split}
\]
Exploiting that $\nu\bigl(B^c(0,1)\bigr)$ is finite since $\nu$ is a L\'evy measure, we can then conclude that $\vert L^{3,2}_m\vert$ tends to zero if $m$ goes to infinity.\newline
The argument for $L^{3,1}_m$ is similar but we need firstly to apply a Taylor expansion twice to make a term $\vert z\vert^2$ appear in the
integral and exploit that $\vert z\vert^2\nu(dz)$ is finite on $B(0,1)$.
%%%%%%%%%%%%%%%%%%%%%%%%%%%%%%%%%%%%%%%%%%%%%%%%%%%%%%%%%%%%%%%%%%%%%%%%%%%%%%%%%%%%%%%%%%%

\emph{Uniqueness.} This proof will follow essentially the same arguments as for Theorem \ref{thm:well_posedness_elliptic}.\newline
Let $u$ be any weak solution of Cauchy problem \eqref{eq:Parabolic_IPDE} with $u_0=f=0$. We are going to show
that $u=0$. \newline
We start considering a mollyfing sequence $\{\rho_m\}_{m \in \N}$ in $C^\infty_c((0,T)\times\R^N)$. Denoting for simplicity $u_m(t,x)=u\ast \rho_m(t,x)$, we then notice that $u_m$ is
continuously differentiable in time and that $u_m(0,x)=0$. It makes sense to define now the function
\begin{equation}\label{Proof:eq:def_fm}
f_m(t,x) \, := \, \partial_t u_m(t,x)-\mathcal{L}^{\text{ou}}u_m(t,x).
\end{equation}
Moreover, we can truncate $f_m$ and $u_m$ if necessary, so that they are integrable with integrable Fourier transform. Then, the same
reasoning in Equations \eqref{proof:Schauder_elliptic1}, \eqref{proof:Schauder_elliptic2} allows us to write that
\[
\begin{cases}
   \partial_t \widehat{u}_{m}(t,\xi)- \Psi^{ou}(\xi)\widehat{u}_m(t,\xi) \, = \, \widehat{f}_{m}(t,\xi), \\
    \widehat{u}_{m}(0,\xi) \, = \, 0.
  \end{cases}
\]
The above equation can be easily solved integrating in time, giving the following representation:
\[\widehat{u}_m(t,\xi) \, = \, \int_{0}^{t} e^{(t-s)\Psi^{ou}(\xi)}
\widehat{f}_m (s,\xi)\, ds.\]
In order to go back to $u_m$, we apply now the inverse Fourier transform to write that
\[u_m(t,x) \, = \, \int_{0}^{t}P_{t-s}f_m(s,x) \, ds.\]
The contraction property of $P_t$  allows us to conclude that $\Vert u_m \Vert_\infty \, \le \, C \Vert f_m\Vert_\infty$.
Letting $m$ goes to zero, we obtain the desired result. Indeed, we can rely on the same reasonings used in the analogous elliptic case
(Theorem \ref{thm:well_posedness_elliptic}) to show that
\[\lim_{m \to \infty}\Vert f_m \Vert_{\infty} \, = \,  0.\qedhere\]
\end{proof}

The next two conclusive theorems provide the Schauder estimates both in the elliptic and in the parabolic setting.

\begin{theorem}[Elliptic Schauder Estimates]
Fixed $\lambda>0$ and $\beta$ in $(0,1)$, let $g$ be in $C^{\alpha+\beta}_{b,d}(\R^N)$. Then, the unique solution $u$ of Equation
\eqref{eq:Elliptic_IPDE} is in $C^{\beta}_{b,d}(\R^N)$ and there exists a
constant $C:=C(\lambda)>0$ such that
\begin{equation}\label{eq:Schauder_Estim_elliptic}
\Vert u \Vert_{C^{\alpha+\beta}_{b,d}} \, \le \, C\Vert g \Vert_{C^{\beta}_{b,d}}.
\end{equation}
\end{theorem}
\begin{proof}
Thanks to Theorem \ref{thm:well_posedness_elliptic}, we know that the unique solution $u$ of the elliptic equation \eqref{eq:Elliptic_IPDE} is given in \eqref{eq:representation_elliptic_sol}. In order to show that such a function $u$ satisfies Schauder estimates
\eqref{eq:Schauder_Estim_elliptic}, we exploit again the equivalent norm defined in \eqref{eq:def_equivalent_norm} of Lemma \ref{lemma:def_equivalent_norm}. Namely, we fix $h$ in $\llbracket 1,n\rrbracket$ and $x_0$ in $\R^N$ and we show that
\[\vert \Delta^3_{x_0}u(z)\vert \, = \, \Bigl{\vert}\int_{0}^{\infty}e^{-\lambda t}\Delta^3_{x_0}\bigl(P_tg\bigr)(z) \, dt\Bigr{\vert} \,
\le \, C\Vert g \Vert_{C^\beta_{b,d}}\vert z\vert^{\frac{\alpha+\beta}{1+\alpha(h-1)}},\quad z\in E_h(\R^N),\]
for some constant $C>0$ independent from $x_0$. For $\vert z \vert\ge 1$, it can be obtained easily from the contraction property of $P_t$
on $B_b(\R^N)$:
\begin{equation}\label{Proof:Schauder:z>1}
\Bigl{\vert}\int_{0}^{\infty}e^{-\lambda t}\Delta^3_{x_0}\bigl(P_tg\bigr)(z) \, dt\Bigr{\vert} \, \le \, C\Vert P_tg\Vert_\infty\, \le \,
C\Vert g \Vert_\infty\vert z \vert^{\frac{\alpha+\beta}{1+\alpha(h-1)}}.
\end{equation}
When $\vert z \vert\le 1$, we start fixing a \emph{transition time} $t_0$ given by
\begin{equation}\label{eq:def:time_t0}
t_0 \:=\, \vert z \vert^{\frac{\alpha}{1+\alpha(h-1)}}.
\end{equation}
Notably, $t_0$ represents the transition time between the diagonal and the off-diagonal regime, accordingly to the intrinsic time scales of the system.
We then decompose $\Delta^3_{x_0}u(z)$ as $R_1(z)+R_2(z)$, where
\begin{align*}
  R_1(z) \, &:= \, \int_{0}^{t_0}e^{-\lambda t}\Delta^3_{x_0}\bigl(P_tg\bigr)(z) \, dt; \\
  R_2(z) \, &:= \, \int_{t_0}^{\infty}e^{-\lambda t}\Delta^3_{x_0}\bigl(P_tg\bigr)(z) \, dt.
\end{align*}
The first component $R_1$ is controlled easily using Corollary \ref{corollary:continuity_between_holder} for $\beta=\gamma$.
Indeed,
\begin{equation}\label{Proof:COntrol_tildeLambda1}
\vert R_1(z)\vert \, \le \, \int_{0}^{t_0}\vert\Delta^3_{x_0}\bigl(P_tg\bigr)(z)\vert \, dt \, \le \, \Vert
P_tg\Vert_{C^{\beta}_{b,d}}\vert z \vert^{\frac{\beta}{1+\alpha(h-1)}}\int_{0}^{t_0}\, dt \, \le \, C\Vert g \Vert_{C^\beta_{b,d}}\vert z
\vert^{\frac{\alpha+\beta}{1+\alpha(h-1)}}.
\end{equation}
On the other hand, the control for $R_2$ can be obtained following Equation \eqref{Proof:eq:Taylor_Expansion} in order to write that
\begin{equation}\label{Proof:COntrol_tildeLambda2}
\begin{split}
\vert R_2(z) \vert \, &\le \,C\Vert g \Vert_{C^\beta_{b,d}}\vert z \vert^3 \int_{t_0}^{\infty}e^{-\lambda
t}\bigl(1+t^{\frac{\beta-3(1+\alpha(h-1))}{\alpha}}\bigr)\, dt \\
&\le C\Vert g \Vert_{C^\beta_{b,d}}\vert z \vert^3\bigl(\lambda^{-1}+\vert
z\vert^{\frac{\alpha+\beta-3(1+\alpha(h-1))}{1+\alpha(h-1)}}\bigr) \\
&\le C\Vert g \Vert_{C^\beta_{b,d}}\vert z \vert^{\frac{\alpha+\beta}{1+\alpha(h-1)}},
\end{split}
\end{equation}
where, in the last step, we exploited that $\vert z \vert \le 1$.
\end{proof}

\begin{theorem}[Parabolic Schauder Estimates]
\label{thm:Parabolic_Schauder_Estimates}
Fixed $T>0$ and $\beta$ in $(0,1)$, let $u_0$ be in $C^{\alpha+\beta}_{b,d}(\R^N)$ and $f$ in
$L^\infty\bigl(0,T;C^{\beta}_{b,d}(\R^N)\bigr)$.
Then, the weak solution $u$ of Cauchy Problem \eqref{eq:Parabolic_IPDE} is in $L^\infty\bigl(0,T;C^{\alpha+\beta}_{b,d}(\R^N)\bigr)$ and
there exists a constant $C:=C(T)>0$ such that
\begin{equation}\label{eq:Schauder_Estimate_parabolic}
\Vert u \Vert_{L^\infty(C^{\alpha+\beta}_{b,d})} \, \le \, C\bigl[\Vert u_0 \Vert_{C^{\alpha+\beta}_{b,d}}+\Vert f
\Vert_{L^\infty(C^{\beta}_{b,d})} \bigr].
\end{equation}
\end{theorem}
\begin{proof}
We are going to show that any function $u$ given by Equation \eqref{eq:representation_parabolic_sol}
satisfies the Schauder Estimates \eqref{eq:Schauder_Estimate_parabolic}. We start splitting the function $u$ in $u_1+u_2$, where
\begin{align}\label{Proof:Decomposition_Parabolic}
  u_1(t,x) &:= P_tu_0(x); \\
  u_2(t,x) &:= \int_{0}^{t}P_{s}f(t-s,x) \, ds.
\end{align}
Corollary \ref{corollary:continuity_between_holder} allows then to control $u_1$ in the
following way:
\[\Vert u_1 \Vert_{L^{\infty}(C^{\alpha+\beta}_{b,d})} \, = \, \sup_{t \in [0,T]}\Vert P_tu_0\Vert_{C^{\alpha+\beta}_{b,d}} \, \le \, C\Vert
u_0 \Vert_{C^{\alpha+\beta}_{b,d}}.\]
In order to deal with the contribution $u_2$, we will follow essentially the same reasoning for the Schauder Estimates in the elliptic
setting. Namely, we use again the equivalent norm defined in \eqref{eq:def_equivalent_norm} of Lemma \ref{lemma:def_equivalent_norm} in order to estimate
\[\Vert u_2\Vert_{L^\infty(C^{\alpha+\beta}_{b,d})} \, \le \, C\Vert f\Vert_{L^\infty(C^{\beta}_{b,d})}.\]
Fixed $h$ in $\llbracket 1,n\rrbracket$ and $x_0$ in $\R^N$, our aim is to show that
\[\vert \Delta^3_{x_0}u_2(z)\vert \, = \, \Bigl{\vert}\int_{0}^{t}\Delta^3_{x_0}\bigl(P_{t-s}f\bigr)(s,z) \, ds\Bigr{\vert} \,
\le \, C\Vert f \Vert_{L^\infty(C^\beta_{b,d}\beta)}\vert z\vert^{\frac{\alpha+\beta}{1+\alpha(h-1)}},\quad z\in E_h(\R^N),\]
for some constant $C>0$ independent from $x_0$. When $\vert z \vert\ge 1$, it can be obtained easily from the contraction property of $P_t$
on $C_b(\R^N)$ as in \eqref{Proof:Schauder:z>1}. For $\vert z \vert\le 1$, we fix again the transition time $t_0$ given in \eqref{eq:def:time_t0}
and we then decompose $\Delta^3_{x_0}u_2(t,z)$ as $\tilde{R}_1(t,z)+\tilde{R}_2(t,z)$, where
\begin{align*}
  \tilde{R}_1(t,z) \, &:= \, \int_{0}^{t\wedge t_0}\Delta^3_{x_0}\bigl(P_{s}f\bigr)(t-s,z) \, ds: \\
  \tilde{R}_2(t,z) \, &:= \, \int_{t\wedge t_0}^{t}\Delta^3_{x_0}\bigl(P_{s}f\bigr)(t-s,z) \, ds.
\end{align*}
The first component $R_1$ can be controlled easily as in \eqref{Proof:COntrol_tildeLambda1}:
\[
\begin{split}
\vert \tilde{R}_1(t,z)\vert \, &\le \, \int_{0}^{t\wedge t_0}\vert\Delta^3_{x_0}\bigl(P_{s}f\bigr)(t-s,z)\vert \, ds \\
&\le \, \vert z \vert^{\frac{\beta}{1+\alpha(h-1)}}\int_{0}^{t\wedge t_0} \Vert P_{s}f(t-s,\cdot)\Vert_{C^{\beta}_{b,d}}\, ds \\
&\le \, C\Vert f \Vert_{L^\infty(C^\beta_{b,d})}\vert z \vert^{\frac{\alpha+\beta}{1+\alpha(h-1)}}.
\end{split}
\]
On the other hand, the control for $R_2$ is obtained following the same steps used in Equation \eqref{Proof:COntrol_tildeLambda2}. Namely,
\[
\begin{split}
\vert \tilde{R}_2(t,z) \vert \, &\le \,C\Vert f \Vert_{L^\infty(C^\beta_{b,d})}\vert z \vert^3 \int_{t\wedge
t_0}^{\infty}\bigl(1+s^{\frac{\beta-3(1+\alpha(h-1))}{\alpha}}\bigr)\, ds \\
&\le C\Vert f \Vert_{L^\infty(C^\beta_{b,d})}\vert z \vert^{\frac{\alpha+\beta}{1+\alpha(h-1)}}. \qedhere
\end{split}
\]
\end{proof}

\setcounter{equation}{0}
\section{Extensions to Time Dependent Operators}
In this final section, we would like to show some possible extensions of our method in order to include more general
operators with non-linear, space-time dependent coefficients. Even in this framework, we will prove the well-posedness of the parabolic
Cauchy problem and show the associated Schauder estimates. \newline
Following \cite{Krylov:Priola10}, our first step is to consider a time-dependent Ornstein-Uhlenbeck operator of the following form:
\begin{multline*}
\mathcal{L}^{\text{ou}}_t\phi(t,x) \, := \\
\frac{1}{2}\text{Tr}\bigl(B_tQB_t^\ast D^2\phi(x)\bigr) +\langle A_tx,D \phi(x)\rangle + \int_{\R^d_0}\bigl[\phi(x+B_tz)-\phi(x)-\langle D_x\phi(x), B_tz\rangle\mathds{1}_{B(0,1)}(z) \bigr] \,\nu(dz),
\end{multline*}
where $B_t:=B\sigma_0(t)$ and $A_t$, $\sigma_0(t)$ are two time-dependent matrixes in $\R^N\otimes \R^N$ and $\R^d\otimes\R^d$, respectively.
From this point further, we assume that the matrixes $A_t$, $\sigma_0(t)$ are measurable in time and that they satisfy the following conditions:
\begin{description}
  \item[{[tK]}] for any fixed $t$ in $[0,T]$, it holds that  $N \, = \, \text{rank}\bigl[B,A_tB,\dots,A^{N-1}_tB\bigr]$;
  \item[{[B]}] the matrix $A_t$ is bounded in time, i.e.\ there exists a constant $\eta>0$ such that
  \[\vert A_t \xi\vert \, \le \,  \eta \vert \xi \vert, \quad \xi \in \R^N;\]
  \item[{[UE]}] the matrix $\sigma_0$ is uniformly elliptic, i.e.\ it holds that
  \[\eta^{-1}\vert \xi \vert^2 \, \le \, \langle \sigma_0(t)\xi,\xi \rangle \, \le \, \eta \vert \xi \vert^2 , \quad (t,\xi) \in [0,T]\times \R^d. \]
\end{description}
It is important to highlight already that this new "time-dependent" version [\textbf{tK}] of the Kalman rank condition [\textbf{K}] allows us to reproduce the same reasonings of Section $2$. In particular, the anisotropic distance $d$ and the Zygmund-H\"older spaces $C^\beta_{b,d}(\R^N)$ can be constructed under these assumptions, even if only at any \emph{fixed} time $t$. A priori, the number of sub-divisions of the space $\R^N$ may change for different times, leading to consider a time-dependent $n(t)$ in Equation \eqref{eq:def_of_n} and, consequently, time-dependent anisotropic distances and H\"older spaces. We will however drop the subscript in $t$ below since it does not add any difficulty in the arguments but it may damage the readability of the article.

\begin{prop}
\label{thm:well_posedness_parabolic_time}
Let $u_0$ be in $C_b(\R^N)$ and $f$ in $L^\infty\bigl(0,T;C_b(\R^N)\bigr)$. Then, there exists a unique solution $u\colon [0,T]\times\R^N\to
\R$ of the following Cauchy problem:
\begin{equation}\label{eq:Parabolic_IPDE:time}
\begin{cases}
  \partial_tu(t,x) \, = \, \mathcal{L}^{\text{ou}}_tu(t,x)+f(t,x), \quad (t,x) \in (0,T)\times \R^N; \\
  u(0,x) \, = \, u_0(x), \quad x \in \R^N.
\end{cases}
\end{equation}
Furthermore, if $u_0$ is in $C^{\alpha+\beta}_{b,d}(\R^N)$ and $f$ in
$L^\infty\bigl(0,T;C^{\beta}_{b,d}(\R^N)\bigr)$, then $u$ is in $L^\infty\bigl(0,T;C^{\alpha+\beta}_{b,d}(\R^N)\bigr)$ and
there exists a constant $C:=C(T,\eta)>0$ such that
\begin{equation}\label{eq:Schauder_Estimate_parabolic:time}
\Vert u \Vert_{L^\infty(C^{\alpha+\beta}_{b,d})} \, \le \, C\bigl[\Vert u_0 \Vert_{C^{\alpha+\beta}_{b,d}}+\Vert f
\Vert_{L^\infty(C^{\beta}_{b,d})} \bigr].
\end{equation}
\end{prop}
\begin{proof}
The proof of this result can be obtained mimicking the arguments already presented in the first part of the article with some slight modifications. The main difference is the introduction of the resolvent $\mathcal{R}_{s,t}$ associated with the matrix $A_t$ in place of the matrix exponential $e^{tA}$. Namely, $\mathcal{R}_{s,t}$ is a time-dependent matrix in $\R^N\otimes \R^N$ that is solution of the following ODE:
\begin{equation}\label{eq:def_resolvent}
\begin{cases}
  \partial_t \mathcal{R}_{s,t} \, = \, A_t \mathcal{R}_{s,t} ,\quad t,s \in [0,T]; \\
  \mathcal{R}_{s,s} \, = \, \Id_{N\times N}.
\end{cases}
  \end{equation}
As said before, Section $2$ follows exactly in the same manner as above except for Lemma \ref{lemma:Decomposition_exp_A} (structure of the resolvent), whose proof can be found in \cite{Huang:Menozzi16}, Lemmas $5.1$ and $5.2$.
The arguments in Section $3$ and $4$ can be applied again, even if the formulation of some objects presented there changes slightly. For example in Equation \eqref{eq:Def_OU_Process}, the N-dimensional Ornstein-Uhlenbeck process $\{X_t\}_{t\ge 0}$ driven by $B_tZ_t$ should be now represented by
\[X_t \, = \, \mathcal{R}_{0,t}x+\int_{0}^{t}\mathcal{R}_{s,t}B_s \, dZ_s, \quad t\ge 0, x \in \R^N.\]
Finally in Section $5$, the uniform ellipticity [\textbf{UE}] of $\sigma_0(t)$ and the boundedness [\textbf{B}] of $A_t$ allow us to control the remainder terms appearing in Equation \eqref{eq:remainder_in_existence} as done above and thus, to conclude as in Theorems \ref{thm:well_posedness_parabolic} and \ref{thm:Parabolic_Schauder_Estimates}.
\end{proof}

Once we have shown our results for the time-dependent Ornstein-Uhlenbeck operator $\mathcal{L}^{\text{ou}}_t$, we add now a non-linearity to the problem, even if only dependent in time. Namely, we are interested in operators of the following form:
\begin{equation}\label{eq:def_operator_Lt}
L_t\phi(t,x) \, := \, \mathcal{L}^{\text{ou}}_t\phi(t,x)+\langle F_0(t),D_x\phi(x)\rangle -c_0(t)\phi(x), \quad (t,x) \in [0,T]\times
\R^N,
\end{equation}
where $c_0\colon [0,T]\to \R$ and $F_0\colon [0,T]\to \R^N$ are two functions. For any sufficiently regular function $\phi\colon[0,T]\to \R$, we are going to denote
\begin{equation}\label{eq:relation_weak_solutions}
\mathcal{T}\phi(t,x) \, := \, e^{-\int_{0}^{t}c_0(s) \,ds}\phi\Bigl(t,x+\int_{0}^{t}F_0(s) \, ds\Bigr), \quad (t,x) \in [0,T]\times \R^N.
\end{equation}
We will see in the next result that the "operator" $\mathcal{T}$ transforms solutions of the Cauchy problem associated with
$\mathcal{L}^\text{ou}_t$ to solutions of the Cauchy problem driven by $L_t$, even if for a modified drift $\mathcal{T}f$.

\begin{lemma}
\label{lemma:relation_weak_solutions}
Fixed $T>0$, let $u_0$ be in $C_b(\R^N)$, $f$ in $L^\infty\bigl(0,T;C_b(\R^N)\bigr)$ and $c_0$, $F_0$ in $C_b([0,T])$. Then, a function
$u\colon [0,T]\times \R^N\to \R$ is a weak solution of Cauchy Problem \eqref{eq:Parabolic_IPDE:time} if and only if the function $v\colon
[0,T]\times \R^N\to \R$ given by $v(t,x)=\mathcal{T}u(t,x)$ is a weak solution of the following Cauchy problem:
\begin{equation}\label{eq:Parabolic_IPDE:ext0}
\begin{cases}
  \partial_tu(t,x) \, = \, L_tu(t,x)+\mathcal{T}f(t,x), \quad (t,x) \in (0,T)\times \R^N; \\
  u(0,x) \, = \, u_0(x), \quad x \in \R^N.
\end{cases}
\end{equation}
In particular, there exists a unique weak solution of Cauchy Problem \eqref{eq:Parabolic_IPDE:ext0}.
\end{lemma}
\begin{proof}
Given a weak solution $u$ of Cauchy problem \eqref{eq:Parabolic_IPDE:time}, we are going to show that the function $v$ given in
\eqref{eq:relation_weak_solutions} is indeed a weak solution of Cauchy Problem \eqref{eq:Parabolic_IPDE:ext0}. The inverse implication  can
be obtained in a similar manner and we will not prove it here. \newline
By mollification if necessary, we can take two sequences $\{c_m\}_{m\in \N}$, $\{F_m\}_{m\in \N}$ in $C^\infty_b([0,T])$ such that $c_m\to c_0$
and $F_m \to F_0$ uniformly in $t$. Furthermore, we denote for simplicity
\[\tilde{c}_m(t) \, := \, \int_{0}^{t}c_m(s) \, ds; \quad \tilde{F}_m(t) \, := \, \int_{0}^{t}F_m(s) \, ds. \]
Given a test function $\phi$ in $C^\infty_c\bigl([0,T)\times \R^N\bigr)$, let us consider for any $m$ in $\N$, the following function
\[\psi_m(t,x) \, := \, e^{-\tilde{c}_{m}(t)}\phi(t,x-\tilde{F}_m(t)) \quad (t,x) \in [0,T]\times\R^N.\]
Since $\tilde{c}_{m}$ and $\tilde{F}_{m}$ are smooth and bounded, it is easy to check that $\psi_m$ is in $C^\infty_c\bigl([0,T)\times
\R^N\bigr)$. We can then use $\psi_m$ in Equation \eqref{eq:weak_solution_Parabolic} (with time-dependent $A_t$ and $B_t$) to show that
\[\int_{0}^{T}\int_{\R^N}\Bigl[\partial_t+\bigl(\mathcal{L}^{\text{ou}}_t\bigr)^*\Bigr]\psi_m(t,y)u(t,y)+f(t,y) \, dydt +
\int_{\R^N}\psi_m(0,y)u_0(y) \, dy \, = \, 0.\]
A direct calculation then show that $\psi_m(0,y)=\phi(0,y)$ and
\begin{align*}
   \bigl(\mathcal{L}^{\text{ou}}_t\bigr)^*\psi_m(t,y) \, &= \,
   e^{-\tilde{c}_m(t)}\bigl(\mathcal{L}^{\text{ou}}_t\bigr)^*\phi(t,y-\tilde{F}_m(t));\\
  \partial_t\psi_m(t,y) \, &= \, e^{-\tilde{c}_m(t)}\Bigl[\partial_t\phi(t,y-\tilde{F}_m(t))-\langle F_m(t),D_y\phi(t,y-\tilde{F}_m(t))
  \rangle-c_m(t)\phi(t,y-\tilde{F}_m(t))\Bigr].
\end{align*}
The above calculations and a change of variable then imply that
\begin{multline*}
\int_{0}^{T}\int_{\R^N}\Bigl[\Bigl(\partial_t+\bigl(\mathcal{L}^{\text{ou}}_t\bigr)^*\Bigr)\phi(t,y)-\langle F_m(t),D_y\phi(t,y) \rangle -
c_m(t)\phi(t,y)\Bigr]\mathcal{T}_mu(t,y)+\phi(t,y)\mathcal{T}_mf(t,y) \, dydt\\
+\int_{\R^N}u_0(y)\phi(0,y) \, dy \, = \, 0,
\end{multline*}
where, analogously to Equation \eqref{eq:relation_weak_solutions}, we have denoted for any function $\varphi\colon [0,T]\times \R^N\to \R$,
\[\mathcal{T}_m\varphi(t,y) \, := \, e^{-\tilde{c}_m(t)}\varphi(t,y+\tilde{F}_m(t)).\]
Following similar arguments exploited in the "existence" part in the proof of Theorem \ref{thm:well_posedness_parabolic}, i.e.\ exploiting
the compact support of $\phi$ and the uniform convergence of the coefficients, it is possible to show that the above expression converges,
when $m$ goes to infinity, to
\[\int_{0}^{T}\int_{\R^N}\Bigl[\partial_t+\bigl(L_t\bigr)^*\Bigr]\phi(t,y)v(t,y)+\mathcal{T}f(t,y) \, dxdt + \int_{\R^N}\phi(0,x)u_0(x) \,
dx]\, = \, 0\]
and thus, that $v$ is a weak solution of Cauchy problem \eqref{eq:Parabolic_IPDE:ext0}.
\end{proof}

Thanks to the previous lemma, we are now able to show the Schauder estimates for the solution $v$ of the Cauchy problem
\eqref{eq:Parabolic_IPDE:ext0} and, more importantly, without changing the constant $C$ appearing in Equation
\eqref{eq:Schauder_Estimate_parabolic:time}.

\begin{prop}
\label{thm:well_posedness_parabolic_Lt}
Fixed $T>0$ and $\beta$ in $(0,1)$, let $u_0$ be in $C^{\alpha+\beta}_{b,d}(\R^N)$, $f$ in $L^\infty\bigl(0,T;C^{\beta}_b(\R^N)\bigr)$ and $c_0$,
$F_0$ in $B_b([0,T])$. Then, the unique solution $v$ of Cauchy Problem \eqref{eq:Parabolic_IPDE:ext0} is in $L^\infty\bigl(0,T;
C^{\alpha+\beta}_b(\R^N)\bigr)$ and it holds that
\begin{equation}\label{eq:Schauder_estimates_Lt}
\Vert v \Vert_{L^\infty(C^{\alpha+\beta}_{b,d})} \, \le \, C\bigl[\Vert u_0 \Vert_{C^{\alpha+\beta}_{b,d}}+\Vert f
\Vert_{L^\infty(C^{\beta}_{b,d})}\bigr],
\end{equation}
where $C:=C(T,\eta)>0$ is the same constant appearing in Theorem \ref{thm:Parabolic_Schauder_Estimates}.
\end{prop}
\begin{proof}
We start denoting for simplicity
\[\tilde{c}_0(t) \, := \, \int_{0}^{t}c_0(s) \, ds \,\, \text{ and }\,\, \tilde{F}_0(t) \, := \, \int_{0}^{t}F_0(s) \, ds.\]
By Lemma \ref{lemma:relation_weak_solutions}, we know that if $v$ is a weak solution of Cauchy problem \eqref{eq:Parabolic_IPDE:ext0}, then
the function
\[u(t,x)\,:=\,e^{\tilde{c}_0(t)}v(t,x-\tilde{F}_0(t))\]
is the weak solution of Cauchy problem \eqref{eq:Parabolic_IPDE:time} with $\tilde{f}$ instead of
$f$, where
\[\tilde{f}(t,x) \, := \, e^{\tilde{c}_0(t)}f(t,x-\tilde{F}_0(t)), \quad (t,x) \in (0,T)\times\R^N.\]
Moreover, we have that $\tilde{f}$ is in $L^\infty\bigl(0,T;C^{\beta}_b(\R^N)\bigr)$. Considering, if necessary, a smaller time interval $[0,t]$ for some $t\le T$, it is not difficult to check from Proposition
\ref{thm:well_posedness_parabolic_time} that
\[\Vert e^{\tilde{c}_0(t)}v(t,\cdot-\tilde{F}_0(t))\Vert_{C^{\alpha+\beta}_{b,d}} \, \le \, C \bigl[\Vert u_0
\Vert_{C^{\alpha+\beta}_{b,d}}+\sup_{s \in [0,t]}\Vert e^{\tilde{c}_0(s)}f(s,\cdot-\tilde{F}_0(t))\bigr].\]
Using now the invariance of the H\"older norm under translations, we can show that
\[
\begin{split}
\Vert v(t,\cdot)\Vert_{C^{\alpha+\beta}_{b,d}} \, &\le \, C \bigl[e^{-\tilde{c}_0(t)}\Vert u_0
\Vert_{C^{\alpha+\beta}_{b,d}}+e^{-\tilde{c}_0(t)}\sup_{s \in [0,t]}\Vert e^{\tilde{c}_0(s)}f(s,\cdot)\bigr]\\
&\le C \bigl[\Vert u_0
\Vert_{C^{\alpha+\beta}_{b,d}}+\sup_{s \in [0,t]}\Vert f(s,\cdot)\bigr],
  \end{split}
\]
where in the last step we exploited that $\tilde{c}_0(t)$ is non-decreasing. Taking the supremum with respect to $t$ on both sides of the above
inequality, we obtain our result.
\end{proof}

\begin{remark}[About space-time dependent coefficients]
We briefly explain here how to extend the Schauder estimates \eqref{eq:Schauder_Estimate_parabolic} to a class of non-linear, space-time
dependent operators, whose coefficients are only locally H\"older continuous in space and may be unbounded. Namely, we are interested in
operators of the following form:
\begin{multline*}
L_{t,x}\phi(t,x) \, := \\
\int_{\R^d_0}\bigl[\phi(x+B\sigma(t,x)z)-\phi(x)-\langle D_x\phi(x), B\sigma(t,x)z\rangle\mathds{1}_{B(0,1)}(z) \bigr] \,\nu(dz)+\langle
F(t,x),D_x\phi(x)\rangle,
\end{multline*}
where $B$ is as in \eqref{eq:Lancon_Pol} and $\sigma \colon [0,T]\times\R^N\to \R^d\otimes\R^d$, $F\colon [0,T]\times\R^N\to \R^N$ are two
measurable functions such that $F(t,0)$ is locally bounded in time and $\sigma$ satisfies assumption [\textbf{UE}] at any fixed $(t,x)$ in
$[0,T]\times \R^N$.\newline
We would like now the operator $L_{t,x}$ to present a similar "dynamical" behaviour as above, i.e.\ the transmission of
the smoothing effect of the L\'evy operator to the degenerate components of the system; see Example \ref{example_basic}. For this reason, we
suppose the following:
\begin{itemize}
  \item the drift $F=(F_1,\dots,F_n)$ is such that for any $i$ in $\llbracket 1,n\rrbracket$, $F_i$ depends only on time and on the last
      $n-(i-2)$ components, i.e.\ $F_i(t, x_{i-1},\dots,x_n)$;
  \item the matrixes $D_{ x_{i-1}}F_i(t,x)$ have full rank $d_i$ at any fixed $(t,x)$ in $[0,T]\times \R^N$.
\end{itemize}
As said before, the functions $F$ and $\sigma$ are assumed to be only locally H\"older in space, uniformly in time. Namely, there exists a
positive constant $K_0$ such that
\begin{equation}\label{eq:Local_Holder}
d\bigl(\sigma(t,x),\sigma(t,y)\bigr) \, \le \, K_0d^\beta(x,y); \quad d\bigl(F_i(t,x),F_i(t,y)\bigr) \, \le \, K_0d^{\beta+\gamma_i}(x,y)
\end{equation}
for any $i$ in $\llbracket 1,n \rrbracket$, any $t$ in $[0,T]$ and any $x,y$ in $\R^N$ such that $d(x,y)\le 1$, where
\begin{equation}\label{Drift_assumptions}
    \gamma_i \,:= \,
    \begin{cases}
        1+ \alpha(i-2), & \mbox{if } i>1; \\
        0, & \mbox{if } i=1.
   \end{cases}
   \end{equation}
We remark in particular that the function $F$ may be unbounded in space.\newline
In order to recover Schauder-type estimates even in this framework, we can follow a perturbative method firstly introduced in
\cite{Krylov:Priola10} that allows to exploit the already proven results for time-dependent operators.
Let us assume for the moment that $\sigma$ and $F$ are \emph{globally} H\"older continuous in space, i.e. they satisfy
\eqref{eq:Local_Holder} for any $x,y$ in $\R^N$. Informally speaking, the method links the
operator $L_{t,x}$ with the space independent operator $L_t$ defined in \eqref{eq:def_operator_Lt}, by "freezing" the coefficients of
$L_{t,x}$ along a \emph{reference path} $\theta\colon [0,T]\to \R^N$ given by
\[\theta_t \, := \, x_0 +\int_{t_0}^{t}F(s,\theta_s) \, ds,\]
for some $(t_0,x_0)$ in $[0,T]\times \R^N$. It is important to highlight that, since $F$ is only H\"older continuous, we need to fix one of
the possible paths satisfying the above dynamics. We point out that the deterministic flow $\theta_t$ associated with the drift $F$ is
introduced precisely to handle the possible unboundedness of $F$. We could then consider a proxy operator $L_t$ whose coefficients are given
by $\sigma_0(t):=\sigma(t,\theta_t)$, $F_0(t):=F(t,\theta_t)$ and
\[\bigl[A_t\bigr]_{i,j} \, = \,
\begin{cases}
  D_{ x_{i-1}}F_i(t,\theta_t), & \mbox{if } j=i-1; \\
  0, & \mbox{otherwise}
\end{cases}
\]
In particular, Theorem \ref{thm:well_posedness_parabolic_Lt} assures the well-posedness and the Schauder estimates for the Cauchy problem
associated with $L_t$.\newline
The final step of the proof would be to expand a solution $u$ of the Cauchy problem associated with $L_{t,x}$ around the proxy $L_t$ through
a Duhamel-like formula and finally show that the expansion error only brings a negligible contribution so that the Schauder estimates still
hold for the original problem.\newline
The a priori estimates for the expansion error are however quite involved (and they are the main reason why we have decided to not show here
the complete proof), since they rely on some non-trivial controls in appropriate Besov norms.\newline
In order to deal with coefficients that are only locally H\"older in space, we need in addition to introduce a "localized" version of the
above reasoning. It would be necessary to multiply a solution $u$ by a suitable bump function $\delta$ that localizes in space along the
deterministic flow $\theta_t$ that characterizes the proxy. Namely, to fix a smooth function $\rho$ that is equal to $1$ on $B(0,1/2)$
and vanishes outside $B(0,1)$ and define $\delta(t,x) := \rho(x-\theta_t)$. We would then follow the above method but with respect to the
"localized" solution
\[v(t,x) \, := \, \delta(t,x)u(t,x), \quad (t,x) \in [0,T]\times \R^N.\]
We suggest the interested reader to see \cite{Chaudru:Honore:Menozzi18_Sharp} for a detailed treatise of the argument in the degenerate
diffusive setting, \cite{Chaudru:Menozzi:Priola19} in the non-degenerate stable framework or \cite{Marino20} for the precise assumptions on
the coefficients.
\end{remark}

\bibliography{bibli}
\bibliographystyle{alpha}
\end{document}